\documentclass[notitlepage,reqno,10pt]{article}

\usepackage{latexsym}
\usepackage{amsmath}
\usepackage{amsfonts}
\usepackage{amssymb}
\usepackage{color}
\usepackage{amsthm}

\newcommand{\ep}{\varepsilon}

\newcommand{\bg }{\bar{g}}
\newcommand{\Sp }{\mathbb{S}^{n-1}}

\renewcommand{\a }{\alpha }
\renewcommand{\b }{\beta }
\renewcommand{\d}{\delta }
\newcommand{\D }{\Delta }

\newcommand{\e }{\varepsilon }
\newcommand{\g }{\gamma}

\newcommand{\n }{\nabla }

\newcommand{\Sig }{\Sigma}
\renewcommand{\t }{\theta }

\newcommand{\Oie}{\Omega_{i,\varepsilon}}
\newcommand{\Oue }{\Omega_{1,\varepsilon}}
\newcommand{\Ode }{\Omega_{2,\varepsilon}}
\newcommand{\Ouei }{\Omega_{1,\varepsilon_{i}}}

\newcommand{\bw}{\overline{w}}
\newcommand{\bwj}{\overline{w}^{j}}
\newcommand{\p}{\partial}
\newcommand{\N}{\mathbb{N}}
\newcommand{\NN}[2]{\mathcal{N}_{#1}(#2)}
\renewcommand{\L}[3]{\mathbb{L}_{#1}(#2)\,[#3]}
\newcommand{\Q}[3]{\mathcal{Q}_{#1}(#2)\,(#3)}

\newcommand{\dt}{\partial_{t}}

\newcommand{\intbar}{\mathop{\int\makebox(-13.5,0){\rule[4pt]{.7em}{0.3pt}}%
\kern-6pt}\nolimits}

\newcommand{\be}{\begin{eqnarray}}
\newcommand{\ee}{\end{eqnarray}}
\newcommand{\hs}{\hspace{1cm}}

\newcommand{\R}{\mathbb{R}}

\newcommand{\bigo}[1]{\mathcal{O} \big( #1 \big)}
\newcommand{\nor}[2]{\|{#1}\|_{#2}}

\def\bea{\begin{eqnarray*}}
\def\eea{\end{eqnarray*}}
\def\f{\frac}

\author{Giovanni CATINO$^{a}$ and Lorenzo MAZZIERI$^{a,b}$}

\date{}

\title{\bf Connected sum construction for $\sigma_{k}$-Yamabe metrics}

\begin{document}

\linespread{1.02}

\parindent=0pt

\newtheorem{lem}{Lemma}[section]
\newtheorem{pro}[lem]{Proposition}
\newtheorem{thm}{Theorem}
\newtheorem{rem}[lem]{Remark}
\newtheorem{cor}[lem]{Corollary}
\newtheorem{df}[lem]{Definition}
\newtheorem{claim}[lem]{Claim}
\newtheorem{conj}[lem]{Conjecture}
\newtheorem{ass}[lem]{Assumption}
\numberwithin{equation}{section}
\newtheorem{ackn}{Acknowledgments\!\!}\renewcommand{\theackn}{}

\maketitle

\begin{center}

\

{\small

\noindent $^a$ SISSA - International School for Advanced Studies

Via Beirut 2-4,

I-34014 Trieste - Italy
}

\

{\small

\noindent $^b$ Max-Planck-Institut f\"ur Gravitationsphysik, Albert-Einstein-Institut

Am M\"uhlenberg 1,

D-14476 Golm - Germany

}

\end{center}

\footnotetext[1]{E-mail addresses: catino@sissa.it, mazzieri@sissa.it}

\begin{center}
{\bf Abstract}

\end{center}

In this paper we produce families of Riemannian metrics with positive constant $\sigma_k$-curvature equal to $2^{-k} {n \choose k}$ by performing the connected sum of two given compact {\em non degenerate} $n$--dimensional solutions $(M_1,g_1)$ and $(M_2,g_2)$ of the (positive) $\sigma_k$-Yamabe problem, provided $2 \leq 2k < n$. The problem is equivalent to solve a second order fully nonlinear elliptic equation.

\begin{center}


\bigskip

\centerline{\bf AMS subject classification:  53C24, 53C20,
53C21, 53C25}

\end{center}

\parindent=0pt

\centerline{}

\vspace{-1cm}


\section{Introduction and statement of the result}\label{s:intro}

In recent years much attention has been given to the study of the Yamabe problem for $\sigma_{k}$--curvature, briefly the $\sigma_{k}$--Yamabe problem. To introduce the analytical formulation, we first recall some background materials from Riemmanian geometry. Given $(M,g)$, a compact Riemannian manifold of dimension $n\geq 3$, we denote respectively by $Ric_g$, $R_g$ the Ricci tensor and the scalar curvature of $(M,g)$. The Schouten tensor of $(M,g)$ is defined as follows  
\begin{eqnarray*}
A_g & := & \tfrac{1}{ n-2} \,\, \big( \, Ric_g \,\, - \,\, \tfrac{1}{2(n-1)} \, R_g g \, \big) \,\, . 
\end{eqnarray*}
If we denote by $\lambda_{1}, \ldots, \lambda_{n}$ the eigenvalues of the symmetric endomorphism $g^{-1}A_{g}$, then the $\sigma_k$-curvature of $(M,g)$ is defined as the $k$-th symmetric elementary function of $\lambda_{1},\ldots,\lambda_{n}$, namely
\begin{eqnarray*}
\sigma_k(g^{-1} A_{g}) \,\,\, := \, \sum_{i_1\, <\,\ldots\,< \, i_k}\lambda_{i_i}\cdot \, \ldots \, \cdot \lambda_{i_k} \,\, \quad \hbox{for $1\leq k \leq n$} \quad \quad  & \hbox{and} & \quad \quad 
\sigma_0 (g^{-1} A_{g}) \,\,\,:= \,\,\,1 .
\end{eqnarray*}
The $\sigma_{k}$--Yamabe problem on $(M,g)$
consists in finding metrics with constant $\sigma_{k}$--curvature in the same conformal class of $g$. The case $k=1$ is the well known Yamabe problem, whose progressive resolution is due to Yamabe \cite{Yamabe}, Trudinger \cite{trudinger}, Aubin \cite{aubin} and Schoen \cite{schoen}. Before presenting the historical overview of the existence results for $k\geq 2$, we need to recall the following notions: a metric $g$ on $M$ is said to be $k$--admissible if it belongs to the $k$--th positive cone $\Gamma^{+}_{k}$, where
$$
g\in\Gamma^{+}_{k}\quad\Longleftrightarrow\quad \sigma_{j}(g^{-1}A_{g})>0\quad\hbox{for}\quad j=1,\ldots,k.
$$
Under the assumption that $g$ is $k$--admissible the $\sigma_{k}$--Yamabe problem has been solved in the case $k=2$, $n=4$ by Chang, Gursky and Yang \cite{cgy1} \cite{cgy2}, for locally conformally flat manifolds by Li and Li \cite{ll} (see also Guan and Wang \cite{gw}), and for $2k>n$ by Gursky and Viaclovsky \cite{gv}. For $2 \leq 2k \leq n$ the problem has been solved by Sheng, Trudinger and Wang \cite{stw} under the extra--hypothesis that the operator is variational. We point out that for 
$k=1,2$ this hypothesis is always fulfilled, whereas for $k\geq 3$ it has been shown in \cite{bg} that this extra assumption is equivalent to the locally conformally flatness. Hence, the (positive) $\sigma_k$--Yamabe problem still remain open for $3 \leq k \leq n/2$ with $(M,g)$ non locally conformally flat. In this optic, our result may eventually be used to produce families of new solutions to this problem showing that it is topologically unobstructed also in the remaining cases. At the end of this section we will give a simple example which will show how to use the connected sum construction to produce a metric in $\Gamma^+_3$ with constant $\sigma_3$--curvature on $(\mathbb{S}^6 \times \mathbb{T}^{2}) \, \sharp \, (\mathbb{S}^6 \times \mathbb{T}^{2})$, which is non locally conformally flat. 

\medskip

To put in perspective our work we briefly recall some results which can be found in literature for connected sum and generalized connected sum of positive scalar curvature metrics, metrics with constant scalar curvature and metrics in the positive cone $\Gamma^+_k$. To fix the notations we recall that the connected sum of two $n$--dimensional Riemannian manifolds $(M_1,g_1)$ and $(M_2,g_2)$ is the topological operation which consists in removing an open ball from 
both $M_1$ and $M_2$ and identifying the leftover boundaries, obtaining a new manifold with possibly different topology. Formally, if $p_i \in M_i$ and for a small enough $\e>0$ we excise the ball $B(p_i, \e)$ from $M_i$, $i=1,2$, the (pointwise) connected sum $M_\e$ of $M_1$ and $M_2$ along $p_1$ and $p_2$ with {\em necksize} $\e$ is the topological manifold defined as
$$
M_{\e}\,\, := \,\, M_{1}\sharp_{\e}M_{2} \,\, = \,\,\left[M_{1}\setminus B(p_{1},\e)\,\cup\, M_{2}\setminus B(p_{2},\e)\right]\big / \sim  \,\, ,
$$
where $\sim$ denotes the identification of the two boundaries $\partial B(p_i,\e)$, $i=1,2$. Of course the new manifold $M_\e$ can be endowed with both a differentiable structure and a metric structure, as it will be explicitly done in Section \ref{s:as}. Even though from a topological point of view the value of the {\em necksize} is forgettable, it will be important to keep track of it when we will deal with the metric structure. The generalized connected sum (or fiber sum) is the same operation where instead of removing tubular neighborhoods of points (i.e., balls), one excises the tubular neighborhood of a submanifold which is embedded in both $M_1$ and $M_2$. 

\medskip

The first issue concerning the interaction between generalized connected sum and the scalar curvature is due to Gromov and Lawson \cite{gl} and Schoen and Yau \cite{sy}. They proved that the generalized connected sum of manifolds with positive scalar curvature metrics performed along submanifolds of codimension at least $3$ can be endowed with a new metric whose scalar curvature is still positive. Later this construction has been extended to the pointwise  connected sum of manifolds carrying $k$--admissible metrics by Guan-Lin-Wang \cite{glw}, under the assumption  $2\leq 2k <n$. As a byproduct of our construction we will be able to reproduce this result, with the additional properties that our metrics have constant $\sigma_k$--curvature and can be chosen as close as desired to the initial metrics $g_1$ and $g_2$. In this sense they may represent a canonical choice among all the possible $k$--admissible metrics on the connected sum manifold. 

\medskip

Concerning the solvability of the Yamabe equation ($k=1$) on the pointwise connected sum of manifolds with constant scalar curvature, we mention the results of Joyce \cite{joyce} for the compact case and Mazzeo, Pollack and Uhlenbeck \cite{mpu} for the non compact case. The generalized connected sum has been treated by the second author in \cite{mazzieri1} and \cite{mazzieri2}. Most part of the geometric features of these issues are common to our construction. 
The main differences come from the analytical nature of the problem. In fact for $k=1$ the equation of interest is a second order semilinear elliptic equation, whereas for $k \geq 2$ the equation becomes fully nonlinear and in general it is not elliptic. To guarantee the ellipticity one has to assume that the (background) metric lies either in the $k$--th positive or in the $k$--th negative cone (for a definition of the $k$--th negative cone $\Gamma^-_k$ see for example \cite{gv2}). Here we just focus on the positive cone case, which for several reasons seems to be the most natural one. In fact the general treatment of the $\sigma_k$--Yamabe problem seems still far to be understood in the negative cone.

\

Before giving the precise statement of our result, we set up the problem and briefly describe the strategy of the proof. Since the aim of our work is to produce metrics with constant positive $\sigma_k$--curvature, it is natural to normalize the constant to be the same as the one of the standard sphere, which is $2^{-k}{n \choose k}$. Hence, we will end up with a family of metrics $\{\widetilde{g}_{\e} \}_\e$ parametrized in terms of the {\em necksize} which satisfy 
\begin{eqnarray}
\label{sigmak=const}
\sigma_k \big( \, \widetilde{g}_\e^{-1}  A_{\widetilde{g}_\e} \big) & = &  2^{-k} \, \hbox{${n \choose k}$} .
\end{eqnarray}
To show the existence of these solutions we start by writing down (see Section \ref{s:as}) an explicit family of approximate solution metrics $\{g_\e \}_\e$ (still parametrized by the {\em necksize}) on $M_\e$. This metrics coincide with $g_i$ on $M_i \setminus B(p_i, \e)$, $i=1,2$, and are close to a model metric on the remaining piece of the connected sum manifold, which in the following will be referred as neck region. The metric which we are going to use as a model in the neck region is described in Section \ref{SCH}. Since it is a complete metric on $\R \times \Sp$ with zero $\sigma_k$--curvature, it yields a natural generalization of the scalar flat Schwarzschild metric. For these reasons we have decided to label it as $\sigma_k$--Schwarzschild metric. The heuristic motivation for choosing this model comes from the fact that for $k=1$ it has been successfully employed in the analogous connected sum constructions for consant scalar curvature metrics, and on the other hand it represents the intrinsic counterpart of the catenoidal neck used in the famous gluing constructions of Kapouleas for constant mean curvature surfaces \cite{k1} \cite{k2}.  

\medskip

The next step in our strategy amounts to look for a suitable correction of the approximate solutions to the desired exact solutions. This will be done by means of a conformal perturbation. At the end it will turn out that for sufficiently small values of the parameter $\e$ such a correction can actually be found together with a very precise control on its size and this will ensure the smooth convergence of the new solutions $\tilde{g}_\e$ to the former metrics $g_i$ on the compact subsets of $M_i \setminus \{p_i\}$, $i=1,2$.

\medskip

Having this picture in mind, we pass now to fix the notations that will be used throughout this paper in order to exploit the conformal perturbative program mentioned above and explained in details in the last part of Section \ref{s:as}. Let $(M,\bar g)$ be a compact smooth $n$--dimensional Riemannian manifold without boundary an let $2 \leq 2k <n$. Taking advantage of this second assumption, we introduce the following formalism for the conformal change 
\begin{eqnarray*}
\bar{g}_u  & := & u^{\frac{4k}{n-2k}}\, \bar{g} ,
\end{eqnarray*}
where the conformal factor $u>0$ is a positive smooth function. In this context $\bar g$ will be referred as the background metric. At a first time the $\sigma_k$--equation for the conformal factor $u$ can be formulated as 
\begin{eqnarray*}
\sigma_k \big(\, \bg_u^{-1} A_{\bg_u} \big) & = &  
2^{-k} \, \hbox{${n \choose k}$}.
\end{eqnarray*}
We recall that the Schouten tensor of $\bg_u$ is related to the one of $A_{\bg}$ by the conformal transformation law
\bea
A_{\bg_{u}} & = & A_{\bg} - 
\tfrac{2k}{n-2k}  u^{-1}{\nabla^2 u}  + 
\tfrac{2kn}{(n-2k)^2} u^{-2}  {du \otimes du}  - 
\tfrac{2k^2}{(n-2k)^2}  u^{-2}  {|du|^2}  \bg  ,
\eea
where $\nabla^2$ and $|\cdot|$ are computed with respect to the background metric $\bg$.
For technical reasons, it is convenient to set
\begin{eqnarray}
\label{endo}
B_{\bg_u}  & := & \tfrac{n-2k}{2k} \,  u^{\frac{2n}{n-2k}} \, \bg_u^{-1} \cdot A_{\bg_u}
\end{eqnarray}
and to reformulate the $\sigma_k$--equation as 
\begin{eqnarray}
\label{eq}
  \mathcal{N}_{\bg} (u)  & := & \sigma_k  \left( B_{\bg_{u}}
  \right)  - \hbox{ ${n \choose k}$} \big( \tfrac{n-2k}{4k} \big)^k u^{\frac{2kn}{n-2k}}  \,\,\,\,   = \,\,\,\,  0 .
\end{eqnarray}
We notice that if two metrics $\bg$ and ${g}$ are related by $\bg=(v/u)^{4k/(n-2k)}{g}$, then the nonlinear operator enjoys the following {\em conformal equivariance property}
\be\label{confequi}
\mathcal{N}_{\bg} \, (u)& = & (v/u)^{-\frac{2kn}{n-2k}} \, \mathcal{N}_{{g}} \, (v).
\ee 
The linearized operator of $\mathcal{N}_{\bg}$ about $u$ is defined as
\be
\label{linope}
\mathbb{L}_{\bg}(u)\,[w]  &:= & \left.
\frac{d}{ds} \right|_{s=0} \mathcal{N}_{\bg}\, (u  + sw).
\ee
This last quantity will play a crucial role in our approach. In fact, as explained in Sections \ref{s:as} and \ref{s:linear}, most part of the analysis in this paper is concerned with the study of the mapping properties of the linearized operator about the approximate solutions $g_\e$'s, that we will write in the form $u_\e^{4k/(n-2k)} \bg$. Here the key point is to provide the linearized operator $\mathbb{L}_{\bg}(u_\e) \, [\,\cdot\,]$ with invertibility and {\em a priori} estimates which are uniform with respect to the parameter $\e$. In fact, if we have this and if the error term $\mathcal{N}_{\bg}(u_\e)$ (which measures the failure of the approximate solutions to be exact solutions) becomes smaller and smaller as $\e \rightarrow 0$, we will be in the position to perform a Newton iteration scheme based on the implicit function theorem, which will finally provide us with a correction $w$ satisfying
$\mathcal{N}_{\bg}(u_\e + w) \, = \,0$. The exact solutions will then be recovered as $\widetilde{g}_\e \,= \, (u_\e + w)^{4k/(n-2k)} \bg$. In order to be able to exploit the linear program (invertibility, {\em a priori} estimates, etc.), it is natural to ask  the linearized operators about the initial metrics to be somehow {\em non degenerate}. The concept of {\em non degeneracy} that we need is made precise in the following 
\begin{df}\label{nondegeneracy}
Let $2 \leq 2k <n$ and suppose that the Riemannan manifold $(M,g)$ is a compact $n$--dimensional and $k$--admissible solution to the (positive) $\sigma_k$--Yamabe problem, in the sense that
$$
g \in \Gamma_k^+ \quad \quad \quad \quad \hbox{and} \quad \quad \quad \quad \NN{g}{1} \,\, = \,\, 0 \quad \quad \hbox{in} \,\, M. 
$$ 
Then $(M,g)$ (as well as the metric $g$) is said to be {\em non degenerate} if
$$  
\mathbb{L}_{g}(1)\,[w]\, = \, 0\quad\hbox{in}\quad M\quad \Longrightarrow \quad w\equiv 0,
$$
where $\L{g}{1}{\,\cdot \,}$ is the linearized operator about the metric $g$.
\end{df}
Our main result reads:
\begin{thm}
\label{main}
Let $(M_{1},g_{1})$ and $(M_{2},g_{2})$ be two compact $n$-dimensional $k$-admissible {\em non degenerate} solutions to the positive $\sigma_k$-Yamabe problem, with $2\leq 2k <n$. Then there exists a positive real number $\e_0 >0$ only depending on $n$, $k$, 
and the $C^2$--norm of the coefficients of 
$g_1$ and $g_2$ such that, for every $\e \in (0, \e_0]$, the connected sum $M_{\e}=M_{1}\sharp_{\e} M_{2}$ can be endowed with a $k$--admissible {\em non degenerate} metric $\widetilde{g}_{\e}$ with constant $\sigma_{k}$--curvature equal to $2^{-k} {n \choose k}$. Moreover $\Vert \widetilde{g}_{\e}- g_{i}\Vert_{C^{r}(K_{i})}\rightarrow 0$ for 
any $r>0$ and 
any compact set $K_{i}\subset M_{i}\setminus\{p_{i}\}$,  the $p_i$'s, $i=1,2$, being the points about which the connect sum is performed.
\end{thm}
We want to point out that the restriction on $k$ in terms of the dimension $n$ perfectly agrees with the hypothesis needed by Guan, Lin and Wang \cite{glw} to prove their gluing result for $k$--admissible metrics. Moreover, the condition $2\leq2k<n$ turns out to be optimal. In fact we will show in Section \ref{obstr} that $\R \mathbb{P}^3$ and $\R \mathbb{P}^4$ with their standard metrics are {\em non degenerate} and $2$-admissible but both the connected sums $\R \mathbb{P}^3 \sharp \, \R \mathbb{P}^3$ and $\R \mathbb{P}^4 \sharp \,\R \mathbb{P}^4$ do not admit any $2$-admissible metric.

\medskip

Some comments are due concerning the {\em non degeneracy} condition introduced in Definition \ref{nondegeneracy}. On one hand this kind of hypothesis is common to all the gluing results based on the implicit function theorem and the perturbative approach (such as the previously mentioned works \cite{joyce}, \cite{mpu}, \cite{mazzieri1} and \cite{mazzieri2}) for the reasons explained above. On the other hand it must be pointed out that this condition is not fulfilled by the standard sphere $\mathbb{S}^n$, since its linearized operator is given by
$$
\mathbb{L}_{\mathbb{S}^n} (1) \, [\,\cdot\,] \,\, = \,\, - \hbox{${n-1 \choose k-1}$} \, \big( \tfrac{n-2k}{4k}\big)^{k-1} \, \big[ \, \Delta_{\mathbb{S}^n}  + \, n  \,\big] \, [\, \cdot \,] \,\, .
$$
This fact will prevent us from using Theorem \ref{main} to attach a sphere to another given solution of the $\sigma_k$--Yamabe problem. However, it is clear that this gluing is not relevant from a topological point of view. A more interesting observation is that, for $k=1$, sequences of spheres can actually be glued together via Schwarzschild--type necks, in order to obtain complete non compact (briefly singular) solutions to the Yamabe problem with isolated singularities on $\mathbb{S}^n$, as it has been done in \cite{schoen2}. For $2\leq 2k <n$ the second author proved in a joint work with Ndiaye \cite{mn} the existence of complete non compact and conformal metrics with constant $\sigma_k$--curvature on $\mathbb{S}^n \setminus \Lambda$ where $\Lambda$ is given by a finite number of points with a symmetric disposition. In this case (as well as in \cite{mp} which is an alternative construction in the case $k=1$) the metrics on the complete ends of the manifold are perturbations of $\sigma_k$--Delaunay metrics (for a definition see \cite{mn}). The $\sigma_k$--Delaunay metrics, on the other hand, are periodic metrics on the cylinder $\R \times Ê\mathbb{S}^{n-1}$ with positive constant $\sigma_k$--curvature, which in an appropriate limit become closer and closer to a sequence of standard $n$--dimensional spheres joined together by means of (infinitely many) $\sigma_k$--Schwarzschild necks. In this sense all the mentioned constructions for the singular problem (\cite{schoen2}, \cite{mp} and \cite{mn}) are consistent.
To conclude this remark about the non compact situation, we recall that for $k=1$ another construction is available and it is the one performed in \cite{mpu}. The solutions provided in this work are the so called dipole metrics, which in other words are connected sum of cylinders $\R \times \Sp$ endowed with $\sigma_1$--Delunay metrics. In a forthcoming paper \cite{cm} we extend this result to $2 \leq 2k <n$ by taking advantage of the connected sum techniques developed in this article.

\medskip

Before proceeding with the rest of the paper, we would like to illustrate with an easy example how Theorem \ref{main} may provide the existence of a nontrivial $k$--admissible metric with constant $\sigma_k$--curvature in one of the cases not covered by the present literature. 

\medskip

\textbf{Example:} $n=8$, $k=3$. Let $(M_i, g_i) = (\mathbb{S}^6 \times \mathbb{T}^2 , g := g_{\mathbb{S}^6} + g_{\mathbb{T}^2})$, $i=1,2$. Clearly this metric is not locally conformally flat and belongs to the $3$ positive cone $\Gamma^{+}_{3}$, since
$$
\sigma_{1}\big(g^{-1}A_{g}\big)=\big(\tfrac{5}{42}\big)\,18,\quad \sigma_{2}\big(g^{-1}A_{g}\big)=\big(\tfrac{5}{42}\big)^{2}\,105,\quad \sigma_{3}\big(g^{-1}A_{g}\big)=\big(\tfrac{5}{42}\big)^{3}\,56.
$$
We verify now that $(\mathbb{S}^6 \times \mathbb{T}^2 , g)$ is {\em non degenearate} in the sense of Definition \ref{nondegeneracy}. Let us assume that $v$ satisfies $\mathbb{L}_{g}(1)\,[v] = 0$. 
A direct computation shows that this is equivalent to 
$$
\big[\,-\Delta_{\mathbb{T}^{2}}\,\,-\,\tfrac{7}{24}\,\Delta_{\mathbb{S}^{6}} \, - \,\tfrac{25}{126}\, \big]\,v \,\, = \,\, 0 .
$$
Using separation of variables, we have the following expansion for $v$
$$
v=\hbox{$\sum_{j=0}^{+\infty}$} \, v^{j}(x)\,\,\phi_{j}(\theta),
$$ 
where $x\in \mathbb{T}^{2}$, $\theta \in \mathbb{S}^{6}$ and $\phi_{j}$ are the eigenfunctions of $\Delta_{\mathbb{S}^{6}}$ satisfying $-\Delta_{\mathbb{S}_{6}}\phi_{j}=\lambda_{j}\,\phi_{j}$ for every $j\in\N$. Hence, we have
$$
-\,\Delta_{\mathbb{T}^{2}}\, v^{j}\,=\, \big[ \,\tfrac{25}{126} \, - \, \tfrac{7}{24}\, \lambda_{j}\,\big]\,v^{j}, \quad j \in \N.
$$
Recalling that $spec(\mathbb{S}^{6})= \{i\,(i+5): \, i\in\N\}$, we have $\lambda_{j}\geq 6$ for $j\geq 1$, which clearly implies $v^{j}\equiv 0$ for $j\geq 1$. On the other hand, for $j=0$ we have
$$
-\,\Delta_{\mathbb{T}^{2}}\, v^{0}\,=\, \tfrac{25}{126} \,v^{0},
$$
but it is well known that the spectrum of the standard flat torus $\mathbb{T}^{2}$ is given by $spec(\mathbb{T}^{2})=\{\,4\pi^{2} i: i\in\N \}$. This implies $v^{0}\equiv0$ and thus the non degeneracy of $(\mathbb{S}^6 \times \mathbb{T}^2 , g)$ is proven. Theorem \ref{main} can now by applied to produce on $(\mathbb{S}^6 \times \mathbb{T}^{2}) \, \sharp \, (\mathbb{S}^6 \times \mathbb{T}^{2})$ a family of constant $\sigma_3$--curvature metrics lying in $\Gamma^+_3$. Moreover, since these metrics are obtained via conformal perturbation of approximate solutions which agree with $g_{i}$ on $M_{i}\setminus B(p_{i},1)$, we conclude that the $\sigma_{3}$--Yamabe metrics produced on $(\mathbb{S}^6 \times \mathbb{T}^{2}) \, \sharp \, (\mathbb{S}^6 \times \mathbb{T}^{2})$ are not locally conformally flat.

\medskip

The plan of the paper is the following: in Section \ref{SCH} we define the $\sigma_{k}$--Schwarzschild metric. In Section \ref{s:as} we construct the approximate solution metrics $\{g_\e\}_\e$ on the connected sum $M_\e$. In Section \ref{s:linear} we provide existence, uniqueness and $\e$--{\em a priori} estimates for solutions to the linearized problem. In Section \ref{s:nonlinear} we  deal with the nonlinear analysis and we will conclude the proof of Theorem \ref{main} by means of a Newton iteration scheme. Finally, in Section \ref{obstr} we  will illustrate with two counterexamples the geometric obstruction which prevent the extension of our gluing theorem to the case $2k\geq n$.

\

\begin{ackn} 
This project started when the second author was a post--doc at the Max-Planck-Institut f\"ur Gravitationsphysik. The authors are partially supported by the Italian project FIRB--IDEAS ``Analysis and Beyond''.
\end{ackn}

\

\section{$\sigma_{k}$--Schwarzschild metric on $\mathbb{R}\times\mathbb{S}^{n-1}$}\label{SCH}

As anticipated in the introduction, the first step in our strategy amounts to build approximate solutions on the connected sum of $(M_{1}, g_{1})$ and $(M_{2}, g_{2})$. To do that we need to change the metric in a neighborhood of the points that we are going to excise, obtaining a new metric in the so called {\em neck region}. In the scalar curvature case a clever choice turns out to be the Schwarzschild metric. This is a complete scalar flat metric conformal to the cylindrical metric $g_{cyl}$ on $\mathbb{R}\times\mathbb{S}^{n-1}$. The explicit formula is given by
$$g:=\cosh\left(\tfrac{n-2}{2}t\right)^{\f{4}{n-2}}g_{cyl}.$$

In a similar way, it is easy to construct a complete conformal metric on $\mathbb{R}\times\mathbb{S}^{n-1}$ with zero $\sigma_{k}$--curvature, for all $2\leq2k<n$. Namely 

\begin{pro} Let $g_{v}$ be a metric on $\mathbb{R}\times\mathbb{S}^{n-1}$ defined by $g_{v}=v^{4k/(n-2k)}g_{cyl}$, $v$ being a positive smooth function depending only on $t\in\mathbb{R}$.  Let us define the quantity $$h(t):=v^{2}(t)-\big(\tfrac{2k}{n-2k}\big)^{2}\dot{v}^{2}(t).$$ Then, if $h_{0}:=h(0)>0$,  the family of positive solutions $v=v(t)$ to the equation
$$
\sigma_{k}(B_{g_{v}})=0\quad\mbox{in }\mathbb{R}\times\mathbb{S}^{n-1}\\
$$
is given by $v(t)=\sqrt{h_{0}}\cosh\left(\tfrac{n-2k}{2k}t-c\right),$ $c\in\mathbb{R}$. 
\end{pro}
\begin{rem} We will refer to the metric defined by $$g_{\Sig}:=v_{\Sig}^{\frac{4k}{n-2k}}g_{cyl},$$
with $v_{\Sig}:=\cosh\left(\tfrac{n-2k}{2k}t\right)$, as the $\sigma_{k}$--Schwarzschild metric.    
\end{rem}
\begin{rem} We notice that the family of solutions $v^{4k/(n-2k)}g_{cyl}$ on $\mathbb{R}\times\mathbb{S}^{n-1}$ obtained in the proposition above give rise to a family of complete conformal radial metrics $u(|x|)^{4k/(n-2k)}g_{\mathbb{R}^{n}}$ on $\mathbb{R}^{n}\setminus\{0\}$ with zero $\sigma_{k}$--curvature via the correspondence $u(|x|)=|x|^{-(n-2k)/2k}v(-\log|x|)$.
\end{rem}

\begin{proof} For convenience the cylindrical metric $g_{cyl}$ will also be denoted by $dt^2+d\theta^2$, where $d\theta^{2}$ is the standard metric on $\mathbb{S}^{n-1}$. Moreover let us denote by $A_{cyl}$ the Schouten tensor of the cylindrical metric. We have for $A_{g_v}$.
\be
\label{conf. Schouten} A_{g_{v}}  =  A_{cyl} - 
\tfrac{2k}{n-2k}  v^{-1}{\nabla^2 v}  + 
\tfrac{2kn}{(n-2k)^2} v^{-2}  {dv \otimes dv}  - 
\tfrac{2k^2}{(n-2k)^2}  v^{-2}  {|dv|^2}  g_{cyl}  ,
\ee
where $\nabla^2$ and $|\cdot|$ are computed with respect to
$g_{cyl}$. Since the Schouten tensor of the cylindrical metric is
explicitly given by
\bea
A_{cyl}  =  -\tfrac{1}{2}  dt^2  + \tfrac{1}{2}
 d\theta^2  .
\eea
From (\ref{endo}) we get
\begin{eqnarray}\label{coeff}
(B_{g_{v}})^t_t & = & -\tfrac{n-2k}{4k}v^2  -  v \ddot{v} 
 +  \tfrac{n-k}{n-2k} \dot{v}^2 \nonumber = -\tfrac{n-2k}{4k}\left(\tfrac{n-k}{k}\right)h+v\left[\left(\tfrac{n-2k}{2k}\right)^{2}v-\ddot{v}\right]\\
(B_{g_{v}})^{i}_j & = & \left( \tfrac{n-2k}{4k} v^2  - 
\tfrac{k}{n-2k} \dot{v}^2 \right)  \delta^{\, i}_j =\tfrac{n-2k}{4k}\,h\,\,\delta^{\,i}_{j}\\
(B_{g_{v}})^{t}_j & = & 0   \,\,= \,\,(B_{g_{v}})^i_t  ,\nonumber
\end{eqnarray}
for $1 \leq i,j \leq n-1$. A straightforward computation yields 
$$\sigma_{k}(B_{g_{v}})=\hbox{${n-1\choose k-1}$}\left(\tfrac{n-2k}{4k}\,h\right)^{k-1}\,v\,\left[\left(\tfrac{n-2k}{2k}\right)^{2}v-\ddot{v}\right].$$ Since $h_{0}:=h(0)>0$ and $v>0$, by continuity, the zero $\sigma_{k}$--curvature equation is equivalent to 
$$
\ddot{v}(t)=\left(\tfrac{n-2k}{2k}\right)^{2}v(t)\quad\mbox{in }\mathbb{R}\\
$$
and $h(t)=h_{0}$ for all $t\in\R$. The statement follows at once. 
\end{proof}

Since in the following we will need to study the mapping properties of the linearized operator about (a (scaled version of) the $\sigma_{k}$--Schwarzschild metric, we consider the conformal perturbation, 
\begin{eqnarray*}
s  \longmapsto  g_s  := (v_{\Sig} + sw)^{\frac{4k}{n-2k}} \, g_{cyl},
\end{eqnarray*}
for $s \in \R$ and $w \in \mathcal{C}^{\infty}\left(\mathbb{R}\times\mathbb{S}^{n-1}\right)$.
Obviously $g_0  =  g_{\Sig}$. Let now $A_{{s}}$ be the Schouten tensor of
the metric $g_s$, and let $B_{{s}}$ be the symmetric $(1,1)$--tensor
defined by
\begin{eqnarray}
\label{def B} B_{{s}}  :=  \tfrac{n-2k}{2k} 
({v_{\Sig}+sw})^{\frac{2n}{n-2k}}  g_{s}^{-1} \cdot A_{{s}} .
\end{eqnarray}
Notice that $A_0  =  A_{g_{\Sig}}$ and $B_0  =  B_{g_{\Sig}}$.

\

Even thought our main purpose is to solve the equation (\ref{eq}), it will be useful to understand the features of the linear operator given by
\bea
\mathbb{L}^{0}_{cyl}(v_{\Sig})[w]:=\left.\frac{d}{ds}\right|_{s=0}  \sigma_k \left( B_{s} \right) .
\eea
To calculate the derivative of $\sigma_k\left( B_s \right)$, we use
the formula
\begin{eqnarray}
\frac{d}{ds}\,\, \sigma_k\left( B_s \right) & = & {\rm tr} \,  \,\,
T_{k-1} \left( B_s \right)  \, \cdot \, \frac{d  B_s}{ds} \,\, ,
\end{eqnarray}
where, for an integer $0 \leq m \leq n$, $T_m \left( B_s \right)$ is
defined as
\begin{eqnarray*}
T_m(B_s) & := & \hbox{$\sum_{j=0}^{m}$} \,\, (-1)^j \,\,\, \sigma_{m-j}(B_s)
\,\, \, B_s^j
\end{eqnarray*}
and it is known as the $m$-th Newton transform of $B_s$ (in the
formula above we use the conventions: $B^0 = \mathbb{I}_n$ and
$\sigma_0(B_s) = 1$). As a consequence we get:
\begin{eqnarray}
\label{formal linearized} \left. \frac{d}{ds}\right|_{s=0}\,\,
\sigma_k\left( B_s \right) & = & \hbox{$\sum_{j=0}^{k-1}$} \,\,(-1)^j \,\,\,
\sigma_{k-1-j} \left( B_0 \right) \,\,\, {\rm tr}  \,\,B^j_0 \cdot
\left. \frac{d B_s}{ds}\right|_{s=0}  \,\, .
\end{eqnarray}
To make the expression above more explicit, we need to compute the
coefficients of $B_s$ and their derivatives at $s=0$. For the coefficients of $B_{0}$, from  formulae \eqref{coeff}, we obtain 
\begin{eqnarray*}
(B_0)^t_t \,\,\, = \,\,\, \tfrac{n-2k}{4k} \tfrac{k-n}{k}\,\, h_{\Sig} \,\, \,\,\quad & \hbox{and} & \quad
\,\, (B_0)^{\,i}_j \,\,\, = \,\,\, \,\tfrac{n-2k}{4k}\,\, h_{\Sig} \,\,
\delta^{\, i}_j \,\, .
\end{eqnarray*}
Replacing $v$ by $v_{\Sig}+sw$ in the identity (\ref{conf.
Schouten}), one can easily obtain the expression for $A_s$. Using
(\ref{def B}) again, it is straightforward to see that:
\begin{eqnarray}
\left. \frac{d \, (B_s)^t_t }{ds}\right|_{s=0} & = & -\,\,v_{\Sig} \,\,
\partial_t^2 w \,\, + \,\, \tfrac{2(n-k)}{n-2k} \, \dot{v_{\Sig}} \,\,
\partial_t w \,\, - \,\, \left(\, \tfrac{n-2k}{2k} \, v_{\Sig} \, + \, \ddot{v_{\Sig}}  \,
\right) \,\, w \nonumber\\
\left. \frac{d \, (B_s)^{\, i}_j }{ds}\right|_{s=0} & = & - \,\, v_{\Sig}
\,\, g_{\, \theta}^{il} \, \left(\nabla^2_\theta \, w \right)_{lj}
\,\, - \,\, \tfrac{2k}{n-2k} \, \dot{v_{\Sig}} \, \delta^{\, i}_j \,\,
\partial_t w \,\, + \,\, \tfrac{n-2k}{2k} \, v_{\Sig} \, \delta^{\,i}_j
\,\, w \\
\left. \frac{d \, (B_s)^{\, t}_j }{ds}\right|_{s=0} & = & - \,\, v_{\Sig}
\,\, \partial_{t}\partial_j w \,\, + \,\, \tfrac{n}{n-2k} \, \dot{v_{\Sig}}
\,\,
\partial_j w \nonumber \\
\left. \frac{d \, (B_s)^{i}_{ t} }{ds}\right|_{s=0} & = & g_{\,
\theta}^{il} \, \left( \, -\,\,v_{\Sig} \,\, \partial_t \partial_l w  \,\,
+ \,\, \tfrac{n}{n-2k} \, \dot{v_{\Sig}} \,\, \partial_l w \, \right)
\nonumber
\end{eqnarray}
This implies, for $1 \leq j \leq k-1$,
\begin{eqnarray*}
{\rm tr}\, \, B_0^{\, j} \,\cdot \,
\left.\frac{dB_s}{ds}\right|_{s=0}  & = & \left( \tfrac{n-2k}{4k}
\right)^j \, h_{\Sig}^j \,\, v_{\Sig} \,\, \left\{ \, \left(\tfrac{k-n}{k}\right)^j \,\, \left[ \, -\,\,
\partial_t^2 w \,\, + \,\, \tfrac{2(n-k)}{n-2k} \, \frac{\dot{v_{\Sig}}}{v_{\Sig}} \,\,
\partial_t w \,\, - \,\, \left( \tfrac{n-2k}{2k} +  \frac{\ddot{v_{\Sig}}}{v_{\Sig}}
\right) \,\, w \,\right] \right. \\
& &  \,\,\,\,\,\,\, \quad  \quad   \quad \,\,\quad \quad
\quad \quad \,\,\,\,\,\,\,\, + \,\,\,\,\, \left[ \, - \,
\Delta_\theta \, w \,\, - \,\, \tfrac{2k(n-1)}{n-2k} \,
\frac{\dot{v_{\Sig}}}{v_{\Sig}} \,\,\,
\partial_t w \,\, + \,\, \tfrac{(n-2k)(n-1)}{2k} \,\, w  \,
\right] \,\, \bigg\}
\end{eqnarray*}
and
\begin{eqnarray}
\label{sigmaj}
\sigma_{k-1-j} (B_0) & = & \left(\tfrac{n-2k}{4k} \right)^{k-1-j} \, h_{\Sig}^{k-1-j}\,\,\tfrac{1+j}{k}\,\hbox{${n\choose k-1-j}$}.
\end{eqnarray}
Using these in the formal expression of the derivative of
$\sigma_k({B_s})$ and using a similar computation as in \cite{mn}, we obtain
\begin{eqnarray}
\label{1st expression D} \mathbb{L}^{0}_{cyl}(v_{\Sig}) [w] = -C_{n,k}\, v_{\Sig}\,h_{\Sig}^{k-1} \big[ \, \partial^{2}_{t}+\tfrac{n-k}{k(n-1)}\D_{\t}-\big(\tfrac{n-2k}{2k}\big)^{2}\big]\, w,
\end{eqnarray}
where
$C_{n,k} \,\, = \,\,  {n-1 \choose k-1} \,\, \left( \tfrac{n-2k}{4k}
\right)^{k-1}$ and $h_\Sig \equiv 1$.

\

Notice that from \eqref{sigmaj} one has immediately that the $\sigma_{k}$--Schwarzschild metric $g_{\Sig}$ belongs to $\overline{\Gamma}_{k}^{+}\cap \Gamma_{k-1}^{+}$, for $2\leq 2k<n$.
\

\section{Approximate solutions and perturbative approach}\label{s:as}

\medskip

In this section we construct the connected sum $M_{\e}:=M_{1}\sharp_{\e} M_{2}$ of the two manifolds $(M_{1},g_{1})$, $(M_{2},g_{2})$ obtained by excising two geodesic balls of radius $\e\in(0,1)$ centered at $p_{1}\in M_{1}$ and $p_{2}\in M_{2}$ and identifying the two left over boundaries. At the same time we will define on $M_{\e}$ a new metric $g_{\e}$ which agrees with the old ones outside the balls of radius one and which is modeled on (a scaled version of) the $\sigma_{k}$--Schwarzschild metric in the neck region.

\

To describe the construction we consider the diffeomorphisms given by the  exponential maps
$$\exp_{p_{i}}: B(O_{p_{i}},1)\subset T_{p_{i}}M_{i}\longrightarrow B(p_{i},1)\subset M_{i}, \quad i=1,2.$$ Next, to fix the notation, we identify the tangent spaces $T_{p_{i}}M_{i}$ with $\R^{n}$. It is well known that this identification yields normal coordinates centered at the points $p_{i}$, namely
$$x: B(p_{1},1)\longrightarrow \R^{n}\quad\quad \mbox{and}\quad\quad y: B(p_{2},1)\longrightarrow \R^{n}.$$
We introduce now asymptotic cylindrical coordinates on the punctured ball $B^{*}(0,1)=x\left(B^{*}(p_{1},1)\right)$ setting $t:=\log\e-\log|x|$ and $\t:=x/|x|$. In this way we have the diffeomorphism $B^{*}(0,1)\simeq(\log\e,+\infty)\times\mathbb{S}^{n-1}$. Analogously, we consider the diffeomorphism $y\left(B^{*}(p_{2},1)\right)=B^{*}(0,1)\simeq(-\infty,-\log\e)\times\mathbb{S}^{n-1}$, this time setting $t:=-\log\e+\log|y|$ and $\t:=y/|y|$. 

\

In order to define the differential structure of $M_{\e}$, we excise a geodesic ball $B(p_{i},\e)$ from $M_{i}$, obtaining an annular region $A(p_{i},1,\e):=B(p_{i},1)\setminus B(p_{i},\e)$, $i=1,2$. The asymptotic cylindrical coordinates introduced above can be used to define a natural coordinate system on the neck region
$$
(t,\t):\,\left[A(p_{1},1,\e)\sqcup A(p_{2},1,\e)\right]/\sim\,\,\longrightarrow (\log\e,-\log\e)\times\mathbb{S}^{n-1}=:T_{\e},
$$
where $\sim$ denotes the equivalence which identifies the boundaries of $B(p_{1},\e)$ and $B(p_{2},\e)$, namely
$$q_{1}\sim q_{2}\,\,\Longleftrightarrow\,\, x/|x|(q_{1})=y/|y|(q_{2})\quad\mbox{and}\quad |x|(q_{1})=\e=|y|(q_{2}).
$$
Clearly, in this coordinates, the two identified boundaries correspond now to the set $\{0\}\times\mathbb{S}^{n-1}$. To complete the definition of the differential structure of the connected sum $M_{\e}$ it is sufficient to consider the old coordinate charts on $M_{i}\setminus B(p_{i},1)$, $i=1,2$. 

\

We are now ready to define on $M_{\e}$ the approximate solution metric $g_{\e}$. First of all we define $g_{\e}$ to be equal to the $g_{i}$ on $M_{i}\setminus B(p_{i},1)$, $i=1,2$. To define $g_{\e}$ in the neck region, we start by observing that the choice of the normal coordinate system allows us to expand the two metric $g_{1}$ and $g_{2}$ around $p_{1}$ and $p_{2}$ respectively as 
$$
g_{1}=\left[\d_{\a\b}+\bigo{|x|^{2}}\right]\,dx^{\a}\otimes dx^{\b}\,\quad\mbox{and}\,\quad g_{2}=\left[\d_{\a\b}+\bigo{|y|^{2}}\right]\,dy^{\a}\otimes dy^{\b}.
$$
In terms of the $(t,\t)$--coordinates we get, for $i=1,2$,
\bea
g_{i}&=& u_{i}^{\frac{4k}{n-2k}}    \left[(1+a^{(i)}_{tt})dt\otimes dt \, + \, (g^{\t}_{jl}+a^{(i)}_{jl})d\t^{j}\otimes d\t^{l} \,
+\,a_{tj}^{(i)}(dt\otimes d\t^{j}+d\t^{j}\otimes dt)\right] ,
\eea
where, as usual, $g^{\t}_{jl}$ are the coefficients of the round metric on $\mathbb{S}^{{n-1}}$, the conformal factors $u_{i}$ are given by
$$
u_{1} := \e^{\frac{n-2k}{2k}} e^{-\frac{n-2k}{2k} t} \quad \mbox{and} \quad u_{2} := \e^{\frac{n-2k}{2k}} e^{\frac{n-2k}{2k} t} 
$$
and, finally, the remainders $a^{(i)}_{\cdot \,\cdot\cdot }$ verify
$$
a^{(1)}_{\cdot \, \cdot\cdot} = \bigo{\e^{2} e^{-2t}} \quad \mbox{and} \quad a^{(2)}_{\cdot \, \cdot\cdot} = \bigo{\e^{2} e^{2t}} .
$$

\

We choose a cut-off functions $\eta: (\log\e, -\log\e)
\rightarrow [0,1]$ to be a non increasing smooth function which is
identically equal to $1$ in $(\log\ep, -1]$ and $0$ in
$[1,-\log\ep)$,  and we choose another cut-off function $\chi :
(\log\ep, -\log\ep) \rightarrow [0,1]$ to be a non increasing smooth
function which is identically equal to $1$ in $(\log\ep, -\log \ep
-1]$ and which satisfies $\lim_{t\rightarrow -\log\ep} \chi = 0$.
Using these two cut-off functions, we can define a new 
conformal factor $u_{\e}$ by
\bea
u_{\e} : = \chi(t) \, u_{1}  + \chi(-t) \,
u_{2}
\eea
and the metric $g_{\e} $ by
\be\label{gepgen}
g_{\e}&=& u_{\e}^{\frac{4k}{n-2k}}    \left[(1+a_{tt})dt\otimes dt \, + \, (g^{\t}_{jl}+a_{jl})d\t^{j}\otimes d\t^{l}
\,+\,a_{tj}(dt\otimes d\t^{j}+d\t^{j}\otimes dt)\right] ,
\ee
where the remeinder $a_{\cdot \,\cdot\cdot }$ verifes
$$
a_{\cdot \, \cdot\cdot} = \eta\, a^{(1)}_{\cdot\,\cdot\cdot}+(1-\eta)\,a^{(2)}_{\cdot\, \cdot\cdot}=\bigo{\e^{2} \cosh(2t)}.
$$
We want to point out that the conformal factor $u_{\e}$ in $(\log\e+1,-\log\e-1)\times\mathbb{S}^{n-1}$ is a scaled version of the conformal factor $v_{\Sig}$ of the $\sigma_{k}$--Schwarzschild metric, namely
$$
u_{\e}(t)=\e^{\frac{n-2k}{2k}}\cosh\left(\tfrac{n-2k}{2k}t\right)
$$ 
in $(\log\e+1,-\log\e-1)\times\mathbb{S}^{n-1}$. In force of this, the approximate solution metric $g_{\e}$ can be viewed as a perturbation of a scaled version of $g_{\Sig}=v_{\Sig}^{4k/(n-2k)}g_{cyl}$, namely
\be
\label{geps}
g_{\e}& = &  \big( 2\,\e^{\frac{n-2k}{2k}}\big)^{\frac{4k}{n-2k}}g_{\Sig}+A^{\e},
\ee
where 
\bea
A^{\e}&=& u_{\e}^{\frac{4k}{n-2k}}    \left[a_{tt}dt\otimes dt + a_{jl}d\t^{j}\otimes d\t^{l}+a_{tj}(dt\otimes d\t^{j}+d\t^{j}\otimes dt)\right],
\eea
and the coefficients verify $A^{\e}_{\cdot\,\cdot\cdot}=u_{\e}^{\frac{4k}{n-2k}} a_{\cdot\,\cdot\cdot}=\bigo{\e^{\frac{n+2k}{2k}}\cosh\big(\tfrac{n+2k}{2k}t\big)}$.

\

To simplify all the computations in the analysis we will make the following assumption

\medskip

\begin{ass}\label{ass} The metric $g_{i}$ is conformally flat in $B(p_{i},1)$, $i=1,2$. 
\end{ass}
\medskip

Later we will show that this assumption can be removed. Now we are going to describe how the expression of $g_{\e}$ simplifies under the conformally flatness of the metric around the gluing locus. First of all, we observe that for $i=1,2$ the metric $g_{i}$ can now be expanded around $p_{i}$ as 
\be\label{asd}
g_{i}=u_{i}^{\frac{4k}{n-2k}}(1+c_{i})\,g_{cyl},
\ee
with $c_{1}=\bigo{\e^{2}e^{-2t}}$ and $c_{2}=\bigo{\e^{2}e^{2t}}$. Thus, it is natural to define the approximate solution metric $g_{\e}$ as 
\be\label{gep}
g_{\e}=u_{\e}^{\frac{4k}{n-2k}}(1+c)\,g_{cyl},
\ee
where $c:=\eta\,c_{1}+(1-\eta)\,c_{2}=\bigo{\e^{2}\cosh(2t)}$. Notice that this definition perfectly agrees with \eqref{gepgen} with 
$$
c_{i}=a^{(i)}_{tt},\quad a^{(i)}_{jl}=c_{i}\,g_{jl}^{\t}\quad\,\mbox{and}\quad\, a^{(i)}_{tj}=0,\,\,\,i=1,2.
$$
To summarize, we fix a background metric $\bar{g}$ defined by 
$$
\bar{g}:=
\begin{cases}
g_{i}\quad\quad&\hbox{on }M_{i}\setminus B(p_{i},1)\\
u_{\e}^{-\frac{4k}{n-2k}}g_{\e}\quad\quad&\hbox{on }A(p_{1},1,\e)\sqcup A(p_{2},1,\e)]/\sim
\end{cases}
$$
In particular we notice that $\bar{g}=(1+c)\,g_{cyl}$ on $A(p_{1},1,\e)\sqcup A(p_{2},1,\e)$ under the Assumption \ref{ass}.
In order to write the approximate solution $g_{\e}$ as a conformal deformation of the background metric $\bar{g}$, it is sufficient to extend the definition of $u_{\e}$ setting
$u_{\e}\equiv1$ on $M_{\e}\setminus T_{\e}$. It is clear that $$g_{\e}=u_{\e}^{\frac{4k}{n-2k}}\bar{g}\,.$$
To conclude the description of the approximate solutions we observe that from our definition it follows immediately that $g_{\e}\rightarrow g_{i}$ with respect to the $C^{m}$--topology on the compact subsets of $M_{i}\setminus\{p_{i}\}$, for $i=1,2$ and every $m\in\mathbb{N}$. A consequence of this fact is the following
\begin{lem}
\label{SGUAN}
Let $g_1$ and $g_2$ be two $(k-1)$--admissible metrics on $M_1$ and $M_2$, respectively.
Then there exists a positive real number $\e_0>0$ only depending on $n$, $k$ and the $C^2$--norm of the coefficients of the metrics $g_1$ and $g_2$ such that, for every $\e \in (0,\e_0]$, the approximate solution $g_\e$ lies in $\Gamma^{+}_{k-1}$.
\end{lem}
\begin{proof}
We argue by contradiction. We fix an index $j \in \{1,\ldots, k-1 \}$ and we suppose that there exists a sequence of parameters $\{\e_i\}_{i\in \mathbb{N}}$ and a corresponding sequence of points $q_i \in M_{\e_i}$, ${i \in \mathbb{N}}$ such that 
\begin{itemize}
\item $\e_i \rightarrow 0$,\quad  as $i \rightarrow +\infty$,
\item $\sigma_j \big( g_{\e_i}^{-1} A_{g_{\e_i}} \big) \, (q_i)\, \leq \, 0 $, \quad  $i \in \N$.
\end{itemize}
Without loss of generality we can suppose, up to pass to a subsequence, that for every $i \in \N$ the point $q_i$ belongs to $M_1 \setminus B(p_1, \e_i)$. Depending on the behavior of the $q_i$'s, we have to distinguish three possible cases.

\medskip

{\bf Case 1:} There exists a subsequence of $q_i$'s such that 
$$
q_{i}\longrightarrow q_{\infty}\in M_{1}\setminus\{p_{1}\}.
$$
Since by construction the metrics $g_{\e_i}$'s  converge to $g_1$ on the compact sets of $M_1 \setminus \{ p_1 \}$ with respect to the $C^2$--norm, we deduce that $\sigma_j \big( g_1^{-1} A_{g_1}\big) \, (q_\infty) \, \leq \, 0$, which is a contradiction.

\medskip

{\bf Case 2:} There exists a subsequence of $q_i$'s such that 
$$
dist_{g_1}(q_{i},p_{1}) \, = \, \bigo{\e_i}, \quad \hbox{as}\quad i\rightarrow+\infty .
$$
Since $q_{i}\rightarrow p_{1}$, the sequence of points $q_{i}$ will stay definitely in the annulus $A(p_{1},1,\e_{i})$, which is mapped to $(\log \e_i , 0 ) \times \mathbb{S}^{n-1}$ via the asymptotic cylindrical coordinates $(t,\t)$. Setting $t_{i}:=t(q_{i})$, we have that, for large enough $i$'s, $(t_i,\t) \in [-C, 0] \times \Sp$, for some fixed positive constant $C>0$.
In this compact region we have that the Schouten tensors $A_{g_{\e_i}}$'s of the approximate solutions $g_{\e_i}$'s converge uniformly to the Schouten tensor $A_{g_\Sig}$ of the $\sigma_k$--Schwarzschild metric $g_\Sig$, according to the conformal transformation law \eqref{conf. Schouten}. The contradiction follows from \eqref{sigmaj}.

\medskip

{\bf Case 3:} There exists a subsequence of $q_i$'s such that 
$$
dist_{g_1}(q_{i},p_{1}) \, = \, {o}\,(\e_i)  \quad \hbox{as}\quad i\rightarrow+\infty .
$$
It is convenient to set $\a_i := dist_{g_1}(q_{i},p_{1}) $. Again, since $q_{i}\rightarrow p_{1}$, the sequence of points $q_{i}$ will stay definitely in the annulus $A(p_{1},1,\e_{i})$, which is mapped to $N_{1,\e_{i}} = (\log \e_i , 0 ) \times \mathbb{S}^{n-1}$ via the asymptotic cylindrical coordinates $(t,\t)$. In this case we have that $t_i = t(q_i) \rightarrow -\infty$, as $i \rightarrow +\infty$. To investigate the behavior of the Schouten tensors $A_{g_{\e_i}}$'s about the points $q_i$'s, it is preferable to translate and rescale all of our quantities by setting 
$$
\hat{g}_{\e_i}(t,\t) \,\, := \,\, \a_i^{-\frac{4k}{n-2k}} \, g_{\e_i}(t+t_i,\t) \,\, = \,\,\big[ \a_i^{-1} \,u_{\e_i}(t+t_i)\big]^{\frac{4k}{n-2k}} \,  (1+ c(t+t_i,\t)) \, g_{cyl}
$$
In terms of these new objects, we have by assumption that $\sigma_j \big(g_{cyl}^{-1} \, A_{\hat{g}_{\e_i}}   \big) (0, \t) \, \leq \, 0$. For any fixed positive constant $C>0$ we have now that on the compact subsets of the form $[-C, C] \times \Sp$ the functions $\a_i^{-1} u_{\e_i}(\, \cdot  \, + t_i)$ converge to $1$ in $C^2$--norm, we deduce that the Schouten tensors $A_{\hat{g}_{\e_i}}$'s converge uniformly to $A_{cyl}$.
A straightforward computation yields
\be
\label{sjc}
\sigma_j \big( g^{-1}_{cyl} \,A_{cyl} \big) & = & 2^{-j} \hbox{${n \choose j}$} \big( \tfrac{n-2j}{n} \big) \,\, > \,\, 0,
\ee
which is a contradiction.
\end{proof}

\


To introduce the analysis which follows, we recall that our
ultimate goal is to show that, up to choose the parameters $\e$ in a suitable range, it is possible to find a smooth perturbation $w$ of the conformal factor $u_{\e}$ such that
\be
\label{perturbation} \mathcal{N}_{\bar{g}}(u_{\e} + w) = 0\,,
\ee
where the nonlinear operator is defined as in \eqref{eq}.

\

As mentioned in the introduction, we want to solve the
fully nonlinear equation (\ref{perturbation}) by means of a fixed
point argument. To do that, we consider the Taylor expansion:
\be\label{equation}
\mathcal{N}_{\bar{g}}(u_{\e} + w) & : = & \mathcal{N}_{\bar{g}}(u_{\e}) \, + \, \mathbb{L}_{\bar{g}}(u_{\e}) [ w] \, + \, \mathcal{Q}_{\bar{g}}(u_{\e}) \,(w) ,
\ee
where according to \eqref{linope}
\be
\label{linearized def.} 
\mathbb{L}_{\bar{g}}(u_{\e})\,[w]  & := &  \left.
\frac{d}{ds} \right|_{s=0} \mathcal{N}_{\bar{g}}\, (u_{\e}  + sw)
\ee
represents the linearized operator of $\mathcal{N}_{\bg}$ around the
approximate solution $u_{\e}$ and
$$
\mathcal{Q}_{\bar{g}}(u_{\e})\,(w) \,\, := \,\, -  \int_0^1 \big[ \, \mathbb{L}_{\bar{g}}(u_{\e}) \, - \, \mathbb{L}_{\bar{g}}(u_{\e} + sw) \, \big]    \, [w] \, ds
$$
is the quadratic remainder.

\



Now, we are ready to study the mapping
properties of $\mathbb{L}_{\bg}(u_{\e})$. In particular we will find the
functional setting where the equation
\be\label{lineareq}
\mathbb{L}_{\bg}(u_{\e}) [w]  =  f\quad\hbox{in}\, M_{\e}
\ee
can be solved with $\e$--uniform {\em a priori} estimates. Combining this
with the estimates of the error term
\bea
\mathcal{E}_{\bg}(u_{\e})(w) & := & - \, \mathcal{N}_{\bg}(u_{\e}) - \mathcal{Q}_{\bg}(u_{\e}) \,(w)  ,
\eea
we will be able to solve the fixed point problem
\begin{eqnarray}
\label{fixed point} w & = & \mathbb{L}_{\bg}(u_{\e})^{-1} \, \circ \,
\mathcal{E}_{\bg}(u_{\e})(w) \,\, .
\end{eqnarray}

\



To solve \eqref{lineareq} we split our domain, namely the connected sum manifold $M_{\e}$, into the regions $\Oie:=M_{i}\setminus B(p_{i},\e)$, $i=1,2$. Notice that $\partial\Oue=\partial\Ode$ by construction. \

As a first step we will produce solutions $w_{i}$ to the Dirichlet problem
\bea
\left\{
\begin{split}
\mathbb{L}_{\bg}(u_{\e})\,[w_{i}] &= f \hs\Omega_{i,\e}\\
w_{i} &= 0 \hs\partial\Omega_{i,\e}
\end{split}
\right.
\eea
Clearly $w_{1}$ and $w_{2}$ have a $C^{0}$--matching on the common boundary $\{0\}\times\mathbb{S}^{n-1}$, but to produce a (weak) global solution on $M_{\e}$, one needs to improve this matching to be at least $C^{1}$. For this purpose we set
\be\label{solution}
w:=\left\{
\begin{split}
w_{1}+\bw_{1} \hs\Oue\\
w_{2}+\bw_{2} \hs\Ode
\end{split}
\right.
\ee
where $\bw_{i}$ are two corrections which verify the homogenous problem
\bea
\left\{
\begin{split}
\mathbb{L}_{\bg}(u_{\e})\,[\bw_{i}] &= 0 \hs\,\,\,\,\Omega_{i,\e}\\
\bw_{i} &= \psi \hs\partial\Omega_{i,\e}
\end{split}
\right.
\eea
with the same Dirichlet boundary data $\psi$ and the $C^{1}$--matching condition
$$\partial_{\nu}(w_{1}+\bw_{1})=-\partial_{\nu}(w_{2}+\bw_{2}),$$ where $\nu$ denotes the outward normal to $\Oue$. We want to point out that in the second part of this program, the datum will be the gap $\partial_{\nu}(w_{1}+w_{2})$ between the normal derivatives of $w_{1}$ and $w_{2}$ and the unknown will be represented by the Dirichlet boundary data $\psi$. The existence of such a function $\psi$ will be deduced from the invertibility of (the difference of) Dirichlet to Neumann maps (for a precise definition see Section \ref{DN}).

\

\

\section{Linear analysis}
\label{s:linear}

The aim of this section is to provide existence, uniqueness and {\em a priori} estimates for solutions to the linear problem
\be\label{eq:global}
\mathbb{L}_{\bg}(u_{\e})\,[w] &=& f \hs M_\e.
\ee
As anticipated in last part of the previous section, we start by dividing the connected sum manifold $M_\e$ into the subdomains $\Oue$ and $\Ode$ and since the situation is  symmetric we will focus for most part of the time on the domain $\Oue$ and we will study the problem
\be\label{eq:local}
\left\{
\begin{split}
\mathbb{L}_{\bg}(u_{\e})\,[w] &= f \hs\Oue\\
w &= 0 \hs\partial\Oue
\end{split}
\right.
\ee
Most part of the work here will amount to establish uniform {\em a priori} estimates for solutions to this problem which do not depend on the necksize parameter $\e$. To do that we will employ a blow--up technique which, in the limit, will lead us to analyze some model situations, depending on where the blow--up points are going to concentrate. As it will be made clear in the proof of Propsition \ref{stime:local2}, in two of the three possible cases, when the blow--up points concentrate on the neck region, we will take advantage of our geometric construction, whereas in the remaining one, when the blow--up points stay away from the gluing locus, we will exclude the blow--up phenomenon thanks to the {\em non degeneracy} condition \ref{nondegeneracy} and Corollary \ref{removable_cor}, which is the main issue of the following subsection.



\

\subsection{A removable singularities lemma}

This subsection, whose content is somehow independent of the rest of the paper, is concerned with the proof of a removable singularities result for the linearized $\sigma_k$--Yamabe equation on the puctured unit ball endowed with a conformally flat metric, see Corollary \ref{removable_cor}.

\

To begin, let $u=u(s,\t)$, $(s,\t)\in(0,+\infty)\times\mathbb{S}^{n-1}$, be a smooth solution to the equation 
$$\mathcal{N}_{{cyl}}(u):= \sigma_k  \left( B_{g_{u}}
  \right)  - \hbox{${n \choose k}$} \left( \tfrac{n-2k}{4k} \right)^k
   u^{\frac{2kn}{n-2k}} =0,$$ on $(0,+\infty)\times\mathbb{S}^{n-1}$. We recall that $g_{u}=u^{\frac{4k}{n-2k}}g_{cyl}$ and 
$$B_{g_{u}}=\tfrac{n-2k}{2k}g_{cyl}^{-1}\left[u^{2} A_{cyl} - 
\tfrac{2k}{n-2k}  u{\nabla^2 u}  + 
\tfrac{2kn}{(n-2k)^2}  {du \otimes du}  - 
\tfrac{2k^2}{(n-2k)^2}  {|du|^2}  g_{cyl}\right].
$$
Suppose to have the expansion $$u(s,\t)=e^{-\frac{n-2k}{2k}s}(1+b_{0}(s,\t)),$$ where $b_{0}(s,\t)=\bigo{e^{-2s}}$. Passing from cylindrical to the flat background metric on the punctured ball, this corresponds to the usual expansion of $g_{u}$ in normal coordinates centered at the removed point.

\begin{lem}\label{exp} Under this hypothesis, if $4\leq 2k<n$, then the solution $u$ verifies
$$u(s,\t)=v(s)(1+b(s,\t)),$$ where $v(s)=\cosh^{-\frac{n-2k}{2k}}(s-s_{0})(1+c(s))$, for some $s_{0}\in\mathbb{R}$, $c(s)=\bigo{e^{-2s}}$ and $b(s,\t)=\bigo{e^{-2ks}}$. 
\end{lem}

\begin{proof} First we observe that
\bea
du&=&\p_{s}u\,\,ds+u(1+b_{0})^{-1}\p_{j}b_{0}\,d\t^{j}=\p_{s}u\,\,ds+u(1+b_{0})^{-1}d_{\t}b_{0},\\
\n^{2}_{cyl}u&=&\p^{2}_{s} u\,ds\otimes ds+\p^{2}_{sj}u\,(ds\otimes d\t^{j}+d\t^{j}\otimes ds)+\n^{2}_{\t}u,
\eea
where $d_{\t}$ and $\n^2_{\t}$ denote the exterior differential and the Hessian computed with respect to the standard differential structure and standard metric of $\mathbb{S}^{n-1}$. Hence, the components of $B_{g_{u}}$ can be written as
\bea
\tfrac{2k}{n-2k}(g_{cyl}B_{g_{u}})_{ss}&=&-\tfrac{1}{2}u^{2}-\tfrac{2k}{n-2k}u\,\p^{2}_{s}u+\tfrac{2k(n-k)}{(n-2k)^{2}}|\p_{s}u|^{2}-\tfrac{2k^{2}}{(n-2k)^{2}}u^{2}(1+b_{0})^{-2}|d_{\t}b_{0}|^{2}\\
&=&-\tfrac{1}{2}u^{2}-\tfrac{2k}{n-2k}u\,\p^{2}_{s}u+\tfrac{2k(n-k)}{(n-2k)^{2}}|\p_{s}u|^{2}+\bigo{u^{2}e^{-4s}},\\
&&\\
\tfrac{2k}{n-2k}(g_{cyl}B_{g_{u}})_{sj}&=&\big(\tfrac{2k}{n-2k}\big)^{2}u\,\p_{s}u(1+b_{0})^{-1}\p_{j}b_{0}+\tfrac{2k}{n-2k}u^{2}(1+b_{0})^{-2}\p_{j}b_{0}\p_{s}b_{0}-\tfrac{2k}{n-2k}u^{2}(1+b_{0})^{-1}\p^{2}_{sj}b_{0}\\
&=&\big(\tfrac{2k}{n-2k}\big)^{2}u\,\p_{s}u\,\p_{j}b_{0}-\tfrac{2k}{n-2k}u^{2}\p^{2}_{sj}b_{0}+\bigo{u^{2}e^{-4s}},\\
&&\\
\tfrac{2k}{n-2k}(g_{cyl}B_{g_{u}})_{ij}&=&\tfrac{1}{2}u^{2}g^{\t}_{ij}-\tfrac{2k}{n-2k}u^{2}(1+b_{0})^{-1}(\n^{2}_{\t}b_{0})_{ij}+\tfrac{2kn}{(n-2k)^{2}}u^{2}(1+b_{0})^{-2}(d_{\t}b_{0}\otimes d_{\t}b_{0})_{ij}\\
&&-\tfrac{2k^{2}}{(n-2k)^{2}}|\p_{s}u|^{2}g^{\t}_{ij}-\tfrac{2k^{2}}{(n-2k)^{2}}u^{2}(1+b_{0})^{-2}|d_{\t}b_{0}|^{2}g^{\t}_{ij}\\
&=&\tfrac{1}{2}\left(u^{2}-\big(\tfrac{2k}{n-2k}\big)^{2}|\p_{s}u|^{2}\right)g^{\t}_{ij}-\tfrac{2k}{n-2k}u^{2}(\n^{2}_{\t}b_{0})_{ij}+\bigo{u^{2}e^{-4s}}\
\eea
For algebraic reasons $\sigma_{k}(B_{g_{u}})$ can be written as
$$\sigma_{k}(B_{g_{u}})\, = \,\hbox{$\sum_{\a l+\b m=k}$}  \,\, C_{\a\b}^{lm}\,\left[tr(B_{g_{u}}^{l})\right]^{\a}\left[tr(B_{g_{u}}^{m})\right]^{\b},$$
where $C_{\a\b}^{lm}$ are 
constant coefficients and we assume that $\a,\b,l,m\in\mathbb{N}$, with $0\leq m,l\leq k$. A direct computation shows that for every $0\leq l \leq k$ 
$$tr(B_{g_{u}}^{l})=[(B_{g_{u}})^{s}_{s}]^{l}+\hbox{$\sum_{i=1}^{n-1}$}\left(B_{g_{u}}^{l}\right)^{i}_{i}+\bigo{u^{2l}e^{-4s}}.$$
Moreover, if we set 
$$h=h(u):=u^{2}-\big(\tfrac{2k}{n-2k}\big)^{2}|\p_{s}u|^{2},$$ we have
\bea2^{l}\,\hbox{$\sum_{i=1}^{n-1}$}  \left(B_{g_{u}}^{\, l}\right)^{i}_{i}&=&h^{l}\,\,tr\left(\mathbb{I}_{n-1}-\tfrac{4k}{n-2k}u^{2}h^{-1}\n^{2}_{\t}b_{0}+\bigo{u^{2}h^{-1}e^{-4s}}\right)^{l}\\
&=&h^{l}\,\,tr\left(\mathbb{I}_{n-1}-\tfrac{4kl}{n-2k}u^{2}h^{-1}\n^{2}_{\t}b_{0}+\bigo{u^{2}h^{-1}e^{-4s}}\right)\\
&=&h^{l}\left((n-1)-\tfrac{4kl}{n-2k}u^{2}h^{-1}\D_{\t}b_{0}\right)+\bigo{u^{2l}e^{-4s}}.
\eea
In force of this considerations we obtain that
\bea 0\,\,\,=\,\,\,N_{{cyl}}(u) &=&\sigma_k  \left( B_{g_{u}}
\right)  - \hbox{${n \choose k}$} \left( \tfrac{n-2k}{4k} \right)^k u^{\frac{2kn}{n-2k}}\\
&=& \big[ A_{n,k}\,u\,(\p_{s}u)^{-1} \big] \cdot  \p_{s}\big(h^{k}-u^{\frac{2kn}{n-2k}}\big) \, + \, \big[ 
\,P_{2k}(u,\p_{s}u) \big] \cdot \D_{\t}b_{0} \, +  \, Q(u,\p u,\p^{2}u),
\eea
where $A_{n,k}$ is a constant only depending on $n$ and $k$, $P_{2k}(\cdot,\cdot\cdot)$ is an homogeneous polynomial of degree $2k$ and the reminder $Q(u,\p u, \p^{2}u)$ verifies the estimate $Q(u,\p u, \p^{2}u)=\bigo{u^{2k}e^{-4s}}$. The gain $e^{-4s}$ is due to the presence of (at least) quadratic terms in $b_{0}$ and its derivatives. Using the eigenfunctions decomposition, we write 
$$
b_{0}(s,\t)=b_{0}^{0}(s)+ \hbox{$\sum_{j=1}^{+\infty}$} \,   b_{0}^{j}(s)\,\phi_{j}(\t)\quad \hbox{and}\quad \D_{\t}b_{0}(s,\t)=- \hbox{ $ \sum_{j=1}^{+\infty} $} \, \lambda_{j}\,b_{0}^{j}(s)\,\phi_{j}(\t),
$$ 
where $- \D_\t \, \phi_j \, = \, \lambda_j \, \phi_j$, $j \in \mathbb{N}$. Since we have
$$
h^{k}=\bigo{e^{-ns}}\quad \hbox{and}\quad u^{\frac{2kn}{n-2k}}=\bigo{e^{-ns}},
$$
$k\geq 2$ and $n\geq 5$, we infer from the equation above that
$$
b_{0}^{j}(s)=\bigo{e^{-4s}},\,\,\,j\geq 1.
$$
So we have found that $u$ expands as
$$
u(s,\t)=v(s)(1+b(s,\t)),
$$
where 
$$v(s):=e^{-\frac{n-2k}{2k}s}(1+b_{0}^{0}(s))\quad \hbox{and} \quad b(s,\t)\, := \, (1+b_{0}^{0}(s))^{-1} \,  \hbox{$\sum_{j=1}^{+\infty}$} \,  b_{0}^{j}(s)\,\phi_{j}(\t)  \, = \, \bigo{e^{-4s}}.$$
We have
\bea
u(\p_{s}u)^{-1} & = & -\tfrac{2k}{n-2k}(1+\bigo{\p_{s}b_{0}^{0}}), \\
h^{k}-u^{\frac{2kn}{n-2k}} & = & h_{v}^{k}-v^{\frac{2kn}{n-2k}}+\bigo{v^{2k}(\p_{s}b_{0}^{0})^{k-1}(\p_{s}b)}+\bigo{v_{1}^{\frac{2kn}{n-2k}}b}, \\
\D_{\t}b_{0} & = & (\D_{\t}b) \, (1+b_{0}^{0}),
\eea
where
$$
h_{v}:=v^{2}-\big(\tfrac{2k}{n-2k}\big)^{2}\dot{v}^{2}.
$$
Combining the new expression for $u$ with the same formal computation used to give a first expansion for the coefficients of $B_{g_{u}}$, we obtain
\be\label{eqsub}
- \big[ \, A_{n,k}\big(\tfrac{2k}{n-2k}\big)\, \big] \cdot \p_{s}\left(h_{v}^{k}-v^{\frac{2kn}{n-2k}}\right)+
\big[\,P_{2k}(v,\p_{s}v) \, \big] \cdot \D_{\t}b \, + \, R \, + \, S \, = \, 0,
\ee
where 
$$R \, = \, \bigo{e^{-(n+2)s}}\, \quad\quad  \hbox{and} \quad \quad \,S \,= \,\bigo{(\D_{\t}b)e^{-{(n-2k+2)}s}}.
$$
This comes from the fact that the leading term of the reminders $R$ and $S$ are given by $\p_{s}b_{0}^{0}\,\p_{s}\big(h_{v}^{k}-v^{\frac{2kn}{n-2k}}\big)$ and $v^{2k}b_{0}^{0}\,\D_{\t}b$ respectively. At this level, we know that $\Delta_{\t}b=\bigo{e^{-4s}}$, so that $S=\bigo{e^{-(n-2k+6)}}$. If $k=2$, then $R$ and $S$ decay 
with the same velocity, namely $\bigo{e^{-(n+2)}}$. For $k>2$, we have $n-2k+6< n+2$. Using again the eigenfunction decomposition, since the first term of the left hand side is radial, we obtain that 
$$b(s,\t)=\bigo{e^{-6s}}.$$
Hence we have obtained an improvement of the expansion for the function $b$ which will improve the estimate of the reminder $S$. Iterating this argument, we will have that, after a finite number of steps, the decay rate of $S$ will be comparable with the one of $R$, which remains fixed during the bootstrap. At the end, both in the cases $k=2$ and $k>2$, we obtain 
$$
b(s,\t)=\bigo{e^{-2ks}}.
$$
Moreover, projecting $\eqref{eqsub}$, we obtain
\be\label{asym}
- \big[ \, A_{n,k}\big(\tfrac{2k}{n-2k}\big) \, \big] \cdot \p_{s}\big(h_{v}^{k}-v^{\frac{2kn}{n-2k}}\big) \,+ \, R^{0} \, + \, S^{0} \, = \, 0,
\ee
where $R^{0}$ and $S^{0}$ are given by
$$
R^{0}(s)\, := \, \int_{\mathbb{S}^{n-1}}R\,\phi_{0}\,\, dV_{\mathbb{S}^{n-1}}\quad \quad  \hbox{and} \quad \quad S^{0}(s) \,:= \,\int_{\mathbb{S}^{n-1}}S\,\phi_{0}\,\, dV_{\mathbb{S}^{n-1}},
$$
and both of them are $\bigo{e^{-(n+2)s}}$. It is now easy to see from equation \eqref{asym} that $v(s)\simeq \cosh^{-\frac{n-2k}{2k}}(s-s_{0})$ for some $s_{0}\in\mathbb{R}$. Plugging the ansatz $v(s)= \cosh^{-\frac{n-2k}{2k}}(s-s_{0})(1+c(s))$ into the equation \eqref{asym} we have that the reminder $c(s)$ is estimated as a $\bigo{e^{-2s}}$.

\end{proof}

We consider now a conformal perturbation of the metric
$g_u$, namely, for $r \in \R$ and $w \in
\mathcal{C}^{2}(\mathbb{R}\times\mathbb{S}^{n-1})$, we consider the
assignment
\begin{eqnarray*}
r & \mapsto & g_r \,\, := \,\, (u + rw)^{\frac{4k}{n-2k}} \cdot \,
dt^2 \otimes \,\, d\theta^2 \,\,.
\end{eqnarray*}
Obviously $g_0  =  g_r$. Let now $A_{{r}}$ be the Schouten tensor of
the metric $g_r$, and let $B_{{r}}$ be the symmetric $(1,1)$-tensor
defined by
\begin{eqnarray*}
B_{{r}} & := & \tfrac{n-2k}{2k} \, \, \,
({u+rw})^{\frac{2n}{n-2k}} \, \, \, g_{r}^{-1} \cdot A_{{r}} \,\,.
\end{eqnarray*}
Again $A_0  =  A_{g_u}$ and $B_0  =  B_{g_u}$. We compute
\begin{eqnarray}\label{lin}
 \quad\quad \quad  \mathbb{L}_{cyl}(u)[w]:= \left.\frac{d}{dr}\right|_{r=0} \mathcal{N}_{{cyl}} (u+rw)  = 
\left.\frac{d}{dr}\right|_{r=0}  \sigma_k \left(  B_{r} \, \right)
 -  \hbox{ ${n \choose k}$} \left( \tfrac{n-2k}{4k}
\right)^k  \left.\frac{d}{dr}\right|_{r=0}
(u+rw)^{\frac{2nk}{n-2k}} 
\end{eqnarray}
To calculate the derivative of $\sigma_k\left( B_r \right)$, we use
the formula
\begin{eqnarray*}
\frac{d}{dr}\,\, \sigma_k\left( B_r \right) & = & {\rm tr} \,  \,\,
T_{k-1} \left( B_r \right)  \, \cdot \, \frac{d  B_r}{dr} \,\, ,
\end{eqnarray*}
where, for an integer $0 \leq h \leq n$, $T_h \left( B_r \right)$ is
defined as
\begin{eqnarray*}
T_h(B_r) & := & \hbox{$\sum_{j=0}^{h}$}\,\, (-1)^j \,\,\, \sigma_{h-j}(B_r)
\,\, \, B_r^j
\end{eqnarray*}
and it is known as the $h$-th Newton transform of $B_r$ (in the
formula above we use the conventions: $B^0 = \mathbb{I}_n$ and
$\sigma_0(B_r) = 1$). As a consequence we get:
\begin{eqnarray}
\label{formlin} \left. \frac{d}{dr}\right|_{r=0}\,\,
\sigma_k\left( B_r \right) & = & \hbox{$\sum_{j=0}^{k-1}$} \,\,(-1)^j \,\,\,
\sigma_{k-1-j} \left( B_0 \right) \,\,\, {\rm tr}  \,\,B^j_0 \cdot
\left. \frac{d B_r}{dr}\right|_{r=0}  \,\, .
\end{eqnarray}

From the previous lemma we know that
$$u(s,\t)=v(s)(1+b(s,\t)),$$ where $v(s)=\cosh^{-\frac{n-2k}{2k}}(s-s_{0})(1+c(s))$, for some $s_{0}\in\mathbb{R}$, $c(s)=\bigo{e^{-2s}}$ and $b(s,\t)=\bigo{e^{-2ks}}$. Hence, the components of $B_{g_{u}}$ can be written as
\bea
\tfrac{2k}{n-2k}(g_{cyl}B_{0})_{ss}&=&-\tfrac{1}{2}u^{2}-\tfrac{2k}{n-2k}u\,\p^{2}_{s}u+\tfrac{2k(n-k)}{(n-2k)^{2}}|\p_{s}u|^{2}-\tfrac{2k^{2}}{(n-2k)^{2}}u^{2}(1+b)^{-2}|d_{\t}b|^{2}\\
&=&-\tfrac{1}{2}v^{2}-\tfrac{2k}{n-2k}v\,\ddot{v}+\tfrac{2k(n-k)}{(n-2k)^{2}}|\dot{v}|^{2}+\bigo{v^{2}e^{-2ks}},\\
& &\\
\tfrac{2k}{n-2k}(g_{cyl}B_{0})_{sj}&=&\big(\tfrac{2k}{n-2k}\big)^{2}u\,\p_{s}u(1+b)^{-1}\p_{j}b+\tfrac{2k}{n-2k}u^{2}(1+b)^{-2}\p_{j}b\p_{s}b-\tfrac{2k}{n-2k}u^{2}(1+b)^{-1}\p^{2}_{sj}b\\
&=&\bigo{v^{2}e^{-2ks}},\\
&&\\
\tfrac{2k}{n-2k}(g_{cyl}B_{0})_{ij}&=&\tfrac{1}{2}u^{2}g^{\t}_{ij}-\tfrac{2k}{n-2k}u^{2}(1+b)^{-1}(\n^{2}_{\t}b)_{ij}+\tfrac{2kn}{(n-2k)^{2}}u^{2}(1+b)^{-2}(d_{\t}b\otimes d_{\t}b)_{ij}\\
&&-\tfrac{2k^{2}}{(n-2k)^{2}}|\p_{s}u|^{2}g^{\t}_{ij}-\tfrac{2k^{2}}{(n-2k)^{2}}u^{2}(1+b)^{-2}|d_{\t}b|^{2}g^{\t}_{ij}\\
&=&\tfrac{1}{2}\left(v^{2}-\big(\tfrac{2k}{n-2k}\big)^{2}|\dot{v}|^{2}\right)g^{\t}_{ij}+\bigo{v^{2}e^{-2ks}}=\tfrac{1}{2}h_{v}\,g^{\t}_{ij}+\bigo{v^{2}e^{-2ks}}\
\eea
It is easy to see that:
\begin{eqnarray*}
\left. \frac{d \, (B_r)^s_s }{dr}\right|_{r=0} & = & -\,\,v \,\,
\partial_s^2 w \,\, + \,\, \tfrac{2(n-k)}{n-2k} \, \dot{v} \,\,
\partial_s w \,\, - \,\, \left(\, \tfrac{n-2k}{2k} \, v \, + \, \ddot{v}  \,
\right) \,\, w +A^s_{s}[w]\nonumber\\
\left. \frac{d \, (B_r)^{\, i}_j }{dr}\right|_{r=0} & = & - \,\, v
\,\, g_{\, \theta}^{il} \, \left(\nabla^2_\theta \, w \right)_{lj}
\,\, - \,\, \tfrac{2k}{n-2k} \, \dot{v} \, \delta^{\, i}_j \,\,
\partial_s w \,\, + \,\, \tfrac{n-2k}{2k} \, v \, \delta^{\,i}_j
\,\, w +A^{i}_{j}[w]\\
\left. \frac{d \, (B_r)^{\, s}_j }{dr}\right|_{r=0} & = & - \,\, v
\,\, \partial_{s}\partial_j w \,\, + \,\, \tfrac{n}{n-2k} \, \dot{v}+
\,\,
\partial_j w +A^{s}_{j}[w]\nonumber \\
\left. \frac{d \, (B_r)^{i}_{ s} }{dr}\right|_{r=0} & = & g_{\,
\theta}^{il} \, \left( \, -\,\,v \,\, \partial_s \partial_l w  \,\,
+ \,\, \tfrac{n}{n-2k} \, \dot{v} \,\, \partial_l w \, \right)+A^{i}_{s}[w],
\nonumber
\end{eqnarray*}
where $A^s_{s},A^{i}_{j},A^{s}_{j},A^{i}_{s}$, $i,j=1,\dots,n-1$, are second order linear operators whose coefficients depend on $v$, $b$ and their derivatives up to order two and can be estimated as $\bigo{v\,e^{-2ks}}$. Hence, from \eqref{formlin}, we obtain that the linearized operator \eqref{lin} splits in
\bea
\mathbb{L}_{cyl}(u)[w]=\mathbb{L}_{cyl}(v)[w]+\mathbb{P}(v,\dot{v},\ddot{v},b,\partial b,\partial^{2}b)[w],
\eea
where $\mathbb{P}$ is a second order linear operators with coefficients estimated by $\bigo{v^{2k-1}e^{-2ks}}$. Now, we recall that from Lemma \ref{exp} we have
$$
v(s)=\bar{v}(s)(1+c(s)),
$$ 
where $\bar{v}(s):=\cosh^{-\frac{n-2k}{2k}}(s-s_{0})$ for some $s_{0}\in\mathbb{R}$ and $c(s)=\bigo{e^{-2s}}$. Hence we can split $\mathbb{L}_{cyl}(v)$ as 
\be\label{split2}
\mathbb{L}_{cyl}(v)[w]=\mathbb{L}_{cyl}(\bar{v})[w]+\mathbb{M}(v,\dot{v},\ddot{v})[w],
\ee
where $\mathbb{M}$ is a second order linear operator with coefficients estimated by $\bigo{e^{-\frac{2kn-n-2k}{2k}s}e^{-2s}}$ and $\mathbb{L}_{cyl}(\bar{v})$ is given by
\be\label{splitconj}
\mathbb{L}_{cyl}(\bar{v})[w]=-C_{n,k}\,\bar{v}^{\frac{k(n-2)}{n-2k}}\,\big(\p^{2}_{s}+\Delta_{\mathbb{S}^{n-1}}-\big(\tfrac{n-2}{2}\big)^{2}+\tfrac{n(n+2)}{4}\bar{v}^{\frac{4k}{n-2k}}\big)\big[\bar{v}^{\frac{n(k-1)}{n-2k}}w\big],
\ee
where $C_{n,k} \,\, = \,\,  {n-1 \choose k-1} \left( \tfrac{n-2k}{4k}
\right)^{k-1}$. Hence it is easy to see that the coefficients of $\mathbb{L}_{cyl}(\bar{v})$ can be estimated by $\bigo{e^{-\frac{2kn-n-2k}{2k}s}}$. As a consequence we have that, for $s\sim+\infty$, the relevant part of the linearized operator $\mathbb{L}_{cyl}(u)$ is given by $\mathbb{L}_{cyl}(\bar{v})$. 

\

In force of these observations we are in the position to prove the following
\begin{lem}\label{removable} Suppose that $u\in C^{2}((0,+\infty)\times\mathbb{S}^{n-1})$ verifies 
$$
\mathcal{N}_{cyl}(u)=0 \quad \hbox{on} \quad (0,+\infty)\times\mathbb{S}^{n-1}
$$
as well as the expansion
$u(s,\t)=e^{-\frac{n-2k}{2k}s}(1+b_{0}(s,\t))$ with $b_{0}(s,\t)=\bigo{e^{-2s}}$.\\
Let $w\in C^{2}((0,+\infty)\times\mathbb{S}^{n-1})$ be such that
$$
\begin{cases}
\hspace{1.94cm}\mathbb{L}_{cyl}(u)[w]&= \,\,\,\,0 \quad  \quad \hbox{on} \quad (0,+\infty)\times\mathbb{S}^{n-1}\\
(\cosh s)^{-\frac{n(k-1)}{2k}}|w(s,\t)| & \leq  \,\,\,\, C\,e^{\delta s},
\end{cases}
$$
for some positive constant $C>0$ and some weight $-\tfrac{n-2}{2}<\delta<\tfrac{n-2}{2}$. Then 
$$
(\cosh s)^{-\frac{n(k-1)}{2k}}|w(s,\t)|\leq C\,e^{-\frac{n-2}{2} \, s}.
$$
\end{lem}

\begin{proof} First of all we notice that the case $k=1$ is well known (see for example \cite{lp}). As we have already seen, for $k\geq 2$, the relevant part of the operator $\mathbb{L}_{cyl}(u)$ at $s\sim +\infty$ is given by $\mathbb{L}_{cyl}(\bar{v})$, defined as above. Hence, the asymptotic behavior of $w$ coincides with the one of a function $\bar{w}$ which satisfies 
$$
\mathbb{L}_{cyl}(\bar{v})[\bar{w}]=0\quad\hbox{on}\quad(0,+\infty)\times\mathbb{S}^{n-1}
$$  
and $(\cosh s)^{-\frac{n(k-1)}{2k}}|\bar{w}(s,\t)|\leq C\,e^{\delta s}$, with $C$ and $\delta$ as above. From \eqref{splitconj} we have that $z:=(\cosh s)^{-\frac{n(k-1)}{2k}}\bar{w}$ satisfies
$$
\big(\p^{2}_{s}+\Delta_{\mathbb{S}^{n-1}}-\big(\tfrac{n-2}{2}\big)^{2}+\tfrac{n(n+2)}{4}\cosh^{-2} (s-s_{0})\big)[z]=0\quad\hbox{on}\quad(0,+\infty)\times\mathbb{S}^{n-1}
$$
and $|z(s,\t)|\leq C\,e^{\delta s}$. Projecting along the eigenfunctions of $\Delta_{\mathbb{S}^{n-1}}$ and using standard ODE's arguments we obtain that
$$
|z(s,\t)|\leq C\,e^{-\frac{n-2}{2}s}.
$$
The statement follows at once.
\end{proof}
Thanks to Lemma \ref{removable} we are now able to prove the following removable singularities result
\begin{cor}\label{removable_cor} Let $g=(1+b_{0})^{\frac{4k}{n-2k}}g_{\mathbb{R}^{n}}$ be a conformally flat metric defined on a geodesic ball $B(p,1)$ verifying the equation
$$\sigma_{k}(g^{-1}A_{g})=2^{-k}\hbox{${n \choose k}$}.$$ Suppose $\bar{w}$ satisfies in the sense of distributions
$$\mathbb{L}_{\mathbb{R}^{n}}(1+b_{0})[\bar{w}]=0\quad\hbox{on}\quad B^{*}(p,1)$$ with $|\bar{w}(q)|\leq C |dist_{g}(q,p)|^{-\mu}$ for any $q\in B^{*}(p,1)$ for some positive constant $C>0$ and for some weight parameter $0<\mu<n-2$. Then $\bar{w}$ is a bounded smooth function on $B(p,1)$ and satisfies the equation above on the entire ball. 
\end{cor}

\begin{proof} First of all we observe that, using normal coordinates centered at $p$, we have $b_{0}(q)=\bigo{|dist_{g}(q,p)|^{2}}$. Passing to cylindrical coordinates and using the conformal equivariance property \eqref{confequi} in order to recover the cylindrical background metric, we have that the equation satisfied by $\bar{w}$ becomes
$$
\mathbb{L}_{cyl}(e^{-\frac{n-2k}{2k}s}(1+b_{0}(s,\t)))[e^{-\frac{n-2k}{2k}s}\bar{w}(s,\t)]=0\quad\hbox{on}\quad(0,+\infty)\times\mathbb{S}^{n-1}.
$$
Letting $w(s,\t):=e^{-\frac{n-2k}{2k}s}\bar{w}(s,\t)$, we deduce from the decay assumption on $\bar{w}$ that
$$
|w(s,\t)|\leq C\,e^{\mu s}e^{-\frac{n-2k}{2k}s}.
$$
If we define $\delta:=\mu-\tfrac{n-2}{2}$, we have that $-\tfrac{n-2}{2}<\delta<\tfrac{n-2}{2}$ and 
$$
|w(s,\t)|\leq C\,e^{\delta s}e^{\frac{n(k-1)}{2k}s}.
$$
We are now in the position to apply Lemma \ref{removable} obtaining $|w(s,\t)|\leq C e^{-\frac{n-2k}{2k}s}$. From this we get
$$
|\bar{w}|\leq C\quad\hbox{on}\quad B^{*}(p,1).
$$
The standard elliptic theory is now sufficient to conclude that $\bar{w}$ can be extended through the point $p$ to a smooth solution on $B(p,1)$.
\end{proof}

\

\subsection{Uniform {\em a priori} estimates on $\Oue$ and $\Ode$}

To state the result, we have to introduce the functional setting. For $m\in\mathbb{N}$ and $\d\in\R$, we consider the weighted $C^{m}_{\d}$--norm defined by
$$
\Vert u\Vert_{C^{m}_{\d}(\Oue)}:=\Vert u \Vert_{C^{m}(M_{1}\setminus B(p_{1},1))}\,+\,\sum_{j=0}^{m}\,\sup_{(t,\t)\in N_{1,\e}}(\e\cosh t)^{\d+j}\,|\nabla^{j}_{g_{\e}}u|_{g_{\e}}(t,\t),
$$ 
where $N_{1,\e}:=(\log\e,0)\times\mathbb{S}^{n-1}$ and the first term is computed with respected to the metric $g_{1}$. Analogously, for $\b\in(0,1)$, we introduce the weighted H\"older $C^{m,\b}_{\d}$--seminorm    
\bea
[\,u\,]_{C^{m,\b}_{\d}(\Oue)}:=[\,u\,]_{C^{m,\b}(M_{1}\setminus B(p_{1},1))}\,+\,  \sup_{(t,\t) \in T_\e} \bigg\{  \, (\e \cosh t)^{\d + m}  \,    
\sup_{(t,\t)\neq(t',\t')} 
\frac{|\n^{m}_{g_{\e}}u(t,\t)-\n_{g_{\e}}^{m}u(t',\t')|_{g_{\e}}}{|dist_{g_{\e}}((t,\t),(t',\t'))|^{\b}}
\bigg\}
,
\eea
where, with the standard convention, the difference between the covariant derivatives is justified up to taking the parallel transport of one of them. The Banach space $C^{m,\b}_{\d}(\Oue)$ is defined by
$$
C^{m,\b}_{\d}(\Oue):=\left\{u\in C^{m,\b}_{loc}(\Oue):\,\,\Vert u\Vert_{C^{m,\b}_{\d}(\Oue)}:=\Vert u\Vert_{C^{m}_{\d}(\Oue)}+[u]_{C^{m,\b}_{\d}(\Oue)}<+\infty\right\}.
$$
We notice that the weighted Banach spaces $C^{m,\b}_{\d}(M_{\e})$, which will be used in the global analysis, can be defined in the same way, replacing $N_{1,\e}$ by $T_{\e}$ and $ M_1 \setminus B_{p_1, 1/2}$ by $\big(M_1 \setminus B_{p_1, 1/2}\big) \, \cup \, \big(M_2 \setminus B_{p_2, 1/2} \big)$. 
With these definitions at hand we are now ready to prove the uniform {\em a priori} estimate
for solutions to the linear problem \eqref{eq:local} on $\Oue$.
\begin{pro}\label{stime:local2} 
Suppose that $\d\in\left(-\f{n-2k}{2k},\f{n-2k}{2k}\right)$ and let $w\in C^{2,\b} (\Oue)$ and $f\in C^{\,0 , \b}(\Oue)$ be two functions satisfying 
\bea
\left\{
\begin{split}
\mathbb{L}_{\bg}(u_{\e})\,[w] &= f \hs\Oue\\
w &= 0 \hs\partial\Oue
\end{split}
\right.
\eea
Then there exist $C=C(n,k,\d,\b)>0$ and $\e_{0}=\e_{0}(n,k,\d)$ such that, for every $\e\in(0,\e_{0}]$, we have
\bea
 \Vert w\Vert_{C^{2,\b}_{\d}(\Oue)}\leq  C\,\Vert f\Vert_{C^{\,0, \b}_{ \d - \frac{n-2k}{2k} (2k-1) }(\Oue)}\,.
\eea
\end{pro}
\begin{proof} 
Here we just provide the uniform weighted $C^0$--bound
\be
 \label{C0bound} 
 \Vert w\Vert_{C^{0}_{\d}(\Oue)}\leq  C\,\Vert f\Vert_{C^{\,0}_{ \d - \frac{n-2k}{2k} (2k-1) }(\Oue)}\,,
\ee
since the uniform weighted $C^{2,\b}$--bound will follows easily from sandard scaling argument, see \cite{pacard}. 

\

To prove \eqref{C0bound} we argue by contradiction. Suppose that there exist a sequence $(\e_{i}, w_{i}, f_{i})$ such that
\begin{itemize}
\item $\e_{i} \, \longrightarrow  \, 0, \quad \mbox{ as } i\rightarrow + \infty,$
\item $\Vert w_{i}\Vert_{C^{0}_{\d}(\Ouei)} \equiv 1 ,\quad \,\, i \in \mathbb{N}$, 
\item $ \Vert f_{i}\Vert_{C^{\,0}_{\d  -  \frac{n-2k}{2k} (2k-1)}(\Ouei)} \longrightarrow 0 ,\quad \mbox{ as } i\rightarrow + \infty$
\end{itemize}
and 
\bea
\left\{
\begin{split}
\mathbb{L}_{\bg}(u_{\e_{i}})\,[w_{i}] &= f_{i} \hs\Ouei\\
w_{i} &= 0 \hs\partial\Ouei
\end{split}
\right.
\eea



To simplify the argument we introduce the function
$$
\zeta_{\e}(q):=\left\{
\begin{split}
\e\cosh(t(q)) &\hs B(p_{1},1)\setminus B(p_{1},\e)\\
1 &\hs M_{1}\setminus B(p_{1},1)
\end{split}
\right.
$$
where, according to Section \eqref{s:as}, $t(q):=\log\e-\log(dist_{g_{1}}(q,p_{1}))$.
With this notation we oberve that the weighted norms are equivalent to
$$
\Vert u\Vert_{C^{m,\b}_{\d}(\Oue)}:=\sum_{j=0}^{m}\,\sup_{\Oue}\zeta_{\e}^{\d+j}\,|\nabla^{j}_{g_{\e}}u|_{g_{\e}}\,+\,\sup_{p\neq q} \left|\min\{\zeta_{\e}(p),\zeta_{\e}(q)\}\right|^{\d+m}\frac{|\n^{m}_{g_{\e}}u(p)-\n_{g_{\e}}^{m}u(q)|_{g_{\e}}}{|dist_{g_{\e}}(p,q))|^{\b}}.
$$ 
Now we consider a sequence of points $q_{i}\in\Ouei$, $i \in \mathbb{N}$, where the maximum of the weighted norm of the functions $w_i$ is achieved, i.e., 
$$\zeta_{\e_{i}}^{\d}(q_{i}) \, |w_{i}(q_{i})|=1.$$ 
Depending on the behavior of the $q_i$'s, we have to distinguish three possible cases.

\medskip

{\bf Case 1:} There exists a subsequence of $q_i$'s such that $$q_{i}\longrightarrow q_{\infty}\in M_{1}\setminus\{p_{1}\}.$$

From the conformal equivariance property we have that
$$
f_i = \mathbb{L}_{\bg}(u_{\e_i}) \, [w_i] = u_{\e_i}^{-\frac{2kn}{n-2k}} \mathbb{L}_{g_{\e_i}} (1) \, [u_{\e_i}^{-1} w_i] .
$$
From the fact that the approximate solution metrics $g_{\e_i}$ converge to $g_1$ on the compact subsets of $M_1 \setminus \{ p_1\}$ with respect to the $C^r$--topology, $r \in \mathbb{N}$, and from the standard elliptic regularity theory, we get that the functions $w_{i}$ converge in the $C^{2}$--norm (computed with respect to the metric $g_{1}$) on each compact subsets of $M_{1}\setminus\{p_{1}\}$ to a function $w_{\infty}$ which satisfies
\be\label{eq:limit1}
\mathbb{L}_{g_{1}}(1)[u_{1}^{-1}w_{\infty}]=0\quad\hbox{on}\,\,\,M_{1}\setminus\{p_{1}\},
\ee 
in the sense of distributions.
Moreover, in the limit point, we have $|w_{\infty}(q_{\infty})|>0$, which means that $w_{\infty}$ is nontrivial. Using that $\Vert w_{i}\Vert_{C^{0}_{\d}(\Omega_{1,\e_{i}})}=1$ and passing to the limit on the compact subsets we obtain the estimate 
$$
|w_{\infty}(q)|\leq C |dist_{g_{1}}(p_{1},q)|^{-\d},\hspace{0.5cm}\,q\in B^{*}(p_{1},1),
$$ 
with $\d\in\left(-\tfrac{n-2k}{2k},\tfrac{n-2k}{2k}\right)$ and $C>0$ a positive constant. This is due to the fact that the weighting functions $\zeta_{\e_i}$ are uniformly comparable to the Riemannian $g_1$--distance to $p_1$. We are going to prove that, in force of this latter feature, the function $u_1^{-1} w_{\infty}$ can be extended to a nontrivial solution of (\ref{eq:limit1}) on the whole $M_{1}$. Using the conformal equivariance property \eqref{confequi} and the fact that, thanks to the assumption \ref{ass}, we can always write on $B(p_1,1)$ 
$$
g_{1}=(1+b_{1})^{\frac{4k}{n-2k}}g_{\R^{n}},
$$ 
with $b_{1}(q)=\bigo{|dist_{g_{1}}(q,p_{1})|^{2}}$, we have that equation above implies
$$
\mathbb{L}_{\R^{n}}(1+b_{1})[(1+b_{1})u_{1}^{-1}w_{\infty}]=0\quad\hbox{on}\quad B^{*}(p_{1},1).
$$ 
Recalling that $g_{1}$ solves the $\sigma_{k}$--Yamabe equation, and that 
$$
|(1+b_{1})u_{1}^{-1}w_{\infty}|(q)\leq C|dist_{g_{1}}(q,p_{1})|^{-\frac{n-2k}{2k}-\delta}
$$
we can apply Corollary \ref{removable_cor} to obtain that $u_1^{-1}w_{\infty}$ extends through $p_{1}$ to a nontrivial smooth solution of 
$$
\mathbb{L}_{g_{1}}(1)[u_{1}^{-1}w_{\infty}]=0\quad\hbox{on}\quad M_{1}.
$$
But this contradicts the {\em non degeneracy} of the metric $g_{1}$ on $M_{1}$ according to Definition \ref{nondegeneracy}.

\medskip

{\bf Case 2:} There exists a subsequence of $q_i$'s such that $$q_{i}\longrightarrow q_{\infty}=p_{1},\quad \a_{i}/\e_{i}=\bigo{1}\quad\hbox{as}\quad i\rightarrow+\infty$$
where $\a_{i}:=\zeta_{\e_{i}}(q_{i})\simeq dist_{g_{1}}(q_{i},p_{1})$. Notice that
$\a_{i}/\e_{i}\simeq\cosh(t_{i})$, where $t_{i}:=t(q_{i})$. Since $q_{i}\rightarrow p_{1}$, the sequence of points $q_{i}$ will stay definitely in the annulus $A(p_{1},1,\e_{i})$, which is mapped to $(\log \e_i , 0 ) \times \mathbb{S}^{n-1}$ via the asymptotic cylindrical coordinates $(t,\t)$. For this reason, with abuse of notation, we can say that $w_{i}(q)=w_{i}(t(q),\t(q))$. Hence, we have $|w_{i}(q_{i})|=|w_{i}(t_{i},\t_{i})|=\a_{i}^{-\d}$. So, if we define 
$$
\bw_{i}:=\a_{i}^{\d}w_{i},\,\,\,\, 
$$
we have $|\bw_{i}(t_{i},\t_{i})|=1$, $|\bw_{i}(t,\t)|\leq (\cosh t)^{-\d}$, for all $(t,\t)\in(-\log\e_{i},0]\times\mathbb{S}^{n-1}$. 
For all $C>0$, we observe that the sequence of points $(t_{i},\t_{i})$ will stay definitely in a compact set of the type $[-C,-1]\times\mathbb{S}^{n-1}$ and, up to a subsequence, they converge to a limit point $(t_{\infty},\t_{\infty})$.

\

In order to investigate the limit problem we introduce an auxiliary function $b$ defined on $T_{\e_i} = (\log {\e_i}, -\log \e_i) \times \mathbb{S}^{n-1}$ in such a way that the following identity is satisfied
$$
\bg = (1+b)^{\frac{4k}{n-2k}} g_{cyl} .
$$
It is immediate to verify that $b={\bigo{\e_i^2}}$ on the compact subset of $T_{\e_i}$. From the conformal equivariance property \eqref{confequi} applied to the problem
\bea
\left\{
\begin{split}
\mathbb{L}_{\bg}(u_{\e_i})\,[w_i] &= f_i \hs  (\log \e_i , 0 ) \times \mathbb{S}^{n-1} \\
w &= 0 \hs \,\, \{0\} \times \mathbb{S}^{n-1}
\end{split}
\right.
\eea
we get
$$
\big [ \, \mathbb{L}^0_{cyl} \big( v_\Sig (1+b)\big) \, - \, \e_i^{{2k}} \hbox{${n \choose k}$} \big( \tfrac{n-2k}{4k}\big)^k \tfrac{2kn}{n-2k} \big( v_{\Sig} (1+b) \big)^{\frac{2kn}{n-2k}-1}  \big] \, [ (1+b) \, \bw_i] \, = \, (1+b)^{\frac{2kn}{n-2k}} \e_i^{- \frac{n-2k}{2k} (2k-1)}  \a_i^{\d} f_i
$$
with $\bw_i (0,\theta) = 0 $, for every $i \in \mathbb{N}$ and $\theta \in \mathbb{S}^{n-1}$.
Since by hypothesis we have supposed that 
$$
\nor{ \,f_i}{\mathcal{C}^{\,\, 0}_{\d - \frac{n-2k}{2k} (2k-1) } (\Ouei)} \longrightarrow 0
$$ 
and $\a_i / \e_i = \bigo{1}$, as $i \rightarrow +\infty$, we deduce that the right hand side of the expression above tends to zero in ${C}^0_{loc} \big( (-\infty , 0) \times \mathbb{S}^{n-1}\big)$.
Moreover it is easy to see that the coefficients of the linear operator on the left hand side tends to the ones of 
$$
\mathbb{L}^0_{cyl}(v_\Sig) = -C_{n,k} \, v_\Sig \, \big[ \, \partial^{2}_{t}+\tfrac{n-k}{k(n-1)}\D_{\t}-\big(\tfrac{n-2k}{2k}\big)^{2}\big]
$$
in ${C}^0_{loc} \big( (-\infty , 0) \times \mathbb{S}^{n-1}\big)$.
By elliptic regularity we obtain the convergence of $\bw_{i}$ to a function $\bw_{\infty}$ in ${C}^{2}_{loc}\big( (-\infty,0)\times\mathbb{S}^{n-1} \big)$, which satisfies in the sense of distributions
\be\label{eq:limit2}
\left\{
\begin{split}
\big[ \, \partial^{2}_{t}+\tfrac{n-k}{k(n-1)}\D_{\t}-\big(\tfrac{n-2k}{2k}\big)^{2}\big]\,\bw_{\infty} &= 0 \hs(-\infty,0)\times\mathbb{S}^{n-1}\\
\bw_{\infty} &= 0 \hs\{0\}\times\mathbb{S}^{n-1}
\end{split}
\right.
\ee
Moreover $\bw_{\infty}$ is nontrivial since, in the limit point, $|\bw_{\infty}(t_{\infty},\t_{\infty})|>0$, and clearly verifies the inequality $|\bw_{\infty}(t,\t)|\leq(\cosh t)^{-\d}$. 
Expanding $\bw_{\infty}$ as $$\bw_{\infty}(t,\t)=\sum_{j=0}^{+\infty}\bw_{\infty}^{j}(t)\,\phi_{j}(\t),$$ where $\phi_{j}$ are the eigenfunctions of $\D_{\t}$ satisfying $-\D_{\t}\,\phi_{j}=\lambda_{j}\phi_{j}$, $j\in\mathbb{N}$, we obtain from \eqref{eq:limit2} that the components $\bw^{j}_{\infty}$ are of the form
$$
\bw^j_{\infty} (t) = A e^{-\mu_j t} +  B e^{\mu_j t} ,
$$
where $A,B \in \R$ and
$$
\mu_{j}:=\left[\tfrac{n-k}{k(n-1)}\lambda_{j}+\left(\tfrac{n-2k}{2k}\right)^{2}\right]^{1/2}.
$$ 
Since $\d\in\left(-\tfrac{n-2k}{2k},\tfrac{n-2k}{2k}\right)$ and $\mu_j \geq \tfrac{n-2k}{2k} > |\d|$, we have that $A$ must be zero. On the other hand the boundary condition implies that $B$ must be zero as well. Hence, $\bw_\infty \equiv 0$, which contradicts the nontriviality.


\medskip

{\bf Case 3:} There exists a subsequence such that $$q_{i}\longrightarrow q_{\infty}=p_{1},\quad \a_{i}/\e_{i}\rightarrow+\infty\quad\hbox{as}\quad i\rightarrow+\infty,$$
where $\a_{i}:=\zeta_{\e_{i}}(q_{i})\simeq dist_{g_{1}}(q_{i},p_{1})$ as before. Again, since $q_{i}\rightarrow p_{1}$, the sequence of points $q_{i}$ will stay definitely in the annulus $A(p_{1},1,\e_{i})$, which is mapped to $N_{1,\e_{i}} = (\log \e_i , 0 ) \times \mathbb{S}^{n-1}$ via the asymptotic cylindrical coordinates $(t,\t)$. With the same abuse of notations as in case 2, we have $|w_{i}(q_{i})|=|w_{i}(t_{i},\t_{i})|=\a_{i}^{-\d}$. To keep track of the nontriviality of the functions $w_{i}$ in the limit, it is convenient to set 
$$
\hat{w}_{i}(t,\t):=\a_{i}^{\d}w_{i}(t+t_{i},\t) .
$$ 
Clearly, we have $|\hat{w}_{i}(0,\t_{i})|=1$, $|\hat{w}_{i}(t,\t)|\leq 2 \, (\cosh t)^{-\d}$, for all $(t,\t)\in(\log\e_{i}-t_{i},-t_{i})\times\mathbb{S}^{n-1}$.
To study the limit problem, we first observe that in this case, due to our definitions, we have $t_{i}\rightarrow -\infty$ and $\log\e_{i}-t_{i}\rightarrow -\infty$ as $i\rightarrow+\infty$. Hence, in the limit, the domain becomes $\R	\times\mathbb{S}^{n-1}$. We define 
$$
\hat{u}_{\e_{i}}(t):=u_{\e_{i}}(t+t_{i})\quad
\hbox{and} \quad \hat{g}(t,\t):=[1+\hat b(t,\t)]g_{cyl},
$$ 
where $\hat b (t, \t) := b(t+t_i , \t)$ and we recall that $u_{\e_i} (t+t_i ) = \e^{(n-2k)/2k}\cosh \big( \tfrac{n-2k}{2k}(t+t_i)\big)$ in this region. In particular we have that 
\bea
\left\{
\begin{split}
\mathbb{L}_{\hat{g}}(\hat u_{\e_i})\,[\hat w_i] &=  \a_i^{\d} f_i (\, \cdot \, +t_i , \cdot	\cdot \, )\hs  (\log \e_i -t_i , -t_i) \times \mathbb{S}^{n-1}\\
\hat w_i &= 0 \hs\hs  \,\, \hs\{-t_i \} \times \mathbb{S}^{n-1}
\end{split}
\right.
\eea
Setting $\hat v_i (t) : = (\e_i / \a_i)^{(n-2k)/2k} \cosh \big(  \tfrac{n-2k}{2k} (t+t_i)\big) $  and using the conformal equivariance property \eqref{confequi} we get 
$$
\big [ \, \mathbb{L}^0_{cyl} \big( \hat v_i (1+ \hat b)\big) \, - \, \a_i^{{2k}} \hbox{${n \choose k}$} \big( \tfrac{n-2k}{4k}\big)^k \tfrac{2kn}{n-2k} \big( \hat v_{i} (1+ \hat b) \big)^{\frac{2kn}{n-2k}-1}  \big] \, [ (1+ \hat b) \, \hat w_i] \, = \, (1+ \hat b)^{\frac{2kn}{n-2k}} \a_i^{- \frac{n-2k}{2k} (2k-1) + \d} f_i (\, \cdot \,+ t_i , \cdot \cdot \,)
$$
with $\hat w_i (-t_i,\theta) = 0 $, for every $i \in \mathbb{N}$ and $\theta \in \mathbb{S}^{n-1}$. Since we have that the functions $\hat w_i$ are uniformly far from zero at $t=0$, we are interested in the limit behavior of the coefficients of our problem on the compact subset of $\R\times \mathbb{S}^{n-1}$ of the form $[-C, C] \times \mathbb{S}^{n-1}$. In this type of region it is immediate to verify that $\hat b = \mathcal{O}({\a_i^2})$ and $\hat v_i$ are uniformly converging to 1. Since by hypothesis we have supposed that 
$$
\nor{ \,f_i}{\mathcal{C}^{\,\, 0}_{\d - \frac{n-2k}{2k} (2k-1) } (\Ouei)} \longrightarrow 0,
$$  
we deduce that the right hand side of the expression above tends to zero in ${C}^0_{loc} \big(  \R\times \mathbb{S}^{n-1}\big)$. Again by elliptic regularity we have the convergence of $\hat{w}_{i}$ to a function $\hat{w}_{\infty}$ in ${C}^{2}_{loc}(\R\times\mathbb{S}^{n-1}) $, which satisfies in the sense of distributions
\be\label{eq:limit3}
\mathbb{L}^0_{cyl} (1) [\hat w_\infty] \,\, = \,\,-C_{n,k} \, \big[ \, \partial^{2}_{t}+\tfrac{n-2k +1 }{(n-1)}\D_{\t}- \tfrac{(n-2k)^2}{2k} \big]\,\hat w_{\infty}  \,\, =  \,\, 0 \hs \R \times\mathbb{S}^{n-1}
\ee
Moreover, up to a subsequence, we have that $\t_i \rightarrow \t_\infty \in \mathbb{S}^{n-1}$ and  $ |\hat w_{\infty} (0, \t_\infty)| = 1 $, hence $\hat w_\infty$ is nontrivial and clearly verifies the inequality $|\hat w_{\infty}(t,\t)| \leq  2 (\cosh t)^{-\d}$. Using the separation of variables as in case 2, we have for $\hat w_{\infty}$ the following expansion
$$
\hat w_{\infty}(t,\t)=\sum_{j=0}^{+\infty}\hat w_{\infty}^{j}(t)\,\phi_{j}(\t).
$$ 
Hence, we infer from \eqref{eq:limit3} that the components $\hat w^{j}_{\infty}$ are of the form
$$
\hat w^j_{\infty} (t) = A e^{-\nu_j t} +  B e^{\nu_j t} ,
$$
where $A,B \in \R$ and
$$
\nu_{j}:=\left[\tfrac{n-2k+1}{(n-1)}\lambda_{j}+\tfrac{(n-2k)^2}{2k}\right]^{1/2}.
$$ 
Since $\d\in\left(-\tfrac{n-2k}{2k},\tfrac{n-2k}{2k}\right)$ and $\nu_j \geq \tfrac{n-2k}{\sqrt{2k}}  > |\d|$, we have that both $A$ and $B$ must be zero. Hence, $\hat w_\infty \equiv 0$, which contradicts the nontriviality.
\end{proof}


We point out that thanks to the Fredholm alternative, see \cite{gt}, the previous proposition also provides existence and uniqueness of solutions to problem \eqref{eq:local}, for sufficiently small values of the parameter $\e$.

\

As an easy consequence of Proposition \ref{stime:local2} we get the following 
\begin{cor}
\label{stime:local2h}
Suppose that $\d\in\left(-\f{n-2k}{2k},\f{n-2k}{2k}\right)$ and let $\bw\in C^{2,\b} (\Oue)$ and $\psi\in C^{\,2 , \b}(\partial \Oue)$ be two functions satisfying 
\bea
\left\{
\begin{split}
\mathbb{L}_{\bg}(u_{\e})\,[\bw] &= 0 \hs\Oue\\
w &= \psi \hs\partial\Oue
\end{split}
\right.
\eea
Then there exist $C=C(n,k,\d,\b)>0$ and $\e_{0}=\e_{0}(n,k,\d)$ such that, for every $\e\in(0,\e_{0}]$, we have
\bea
 \Vert \bw \Vert_{C^{2,\b}_{\d}(\Oue)}\leq  C\, \e^\d\Vert \psi\Vert_{C^{\,2, \b}(\partial \Oue)}
 \,.
\eea
\end{cor}
\begin{proof} It is sufficient to observe that it is always possible to define the extension of $\psi$ as $\widetilde{\psi} (t,\t) := \chi(t) \psi(\t)$, where $\chi$ is a smooth nondecreasing cut-off supported in $[-1,0]$ with $\chi(0) = 1$. Now we just apply the previous proposition to the function $v:=\bw-\widetilde{\psi}$. The desired estimate follows from the fact that $\nor{\widetilde \psi}{C^{2,\b}_\d (\Oue)} \leq 2 \e^\d \nor{\psi}{C^{2,\b}(\partial \Oue)}$ by construction.
\end{proof}

\

\subsection{Dirichlet to Neumann map}\label{DN}

We introduce now the Dirichlet
to Neumann map for the operator $\mathbb{L}_{\bg}(u_{\e})$ on $\Oue$. For any Dirichlet data $\psi \in
\mathcal{C}^{\, 2, \beta} (\partial \Oue)$, we consider the problem
\be
\label{eq:localh}
\left\{
\begin{split}
\mathbb{L}_{\bg}(u_{\e})\,\bw &= 0 \hs\Oue\\
\bw &= \psi \hs\partial\Oue
\end{split}
\right.
\ee
Thanks to Corollary \ref{stime:local2h} for $\e$ sufficiently small, we have (uniform) {\em a priori} estimates, existence and uniqueness of a solution $\bw_{1}(\psi)$ to this problem. In force of these considerations, we define the Dirichlet to Neumann map for the problem (\ref{eq:localh}) as
\bea
T_{\e}:\,C^{2,\b}(\mathbb{S}^{n-1})\longrightarrow C^{1,\b}(\mathbb{S}^{n-1}),\hs T_{\e}:\, \psi \longmapsto \dt{\bw_{1}(\psi)}_{\left|_{\partial\Oue}\right.}=:\partial_{\nu}{\bw_{1}(\psi)}_{\left|_{\partial\Oue}\right.},
\eea
where $\partial_{\nu}$ will denote the outward normal derivative to $\Oue$. It follows from the considerations above that this is a well defined linear operator, which is uniformly bounded in $\e$, for $\e$ sufficiently small. The definition can be obviously extended to an operator (denoted in the same manner) acting between $H^{1}(\mathbb{S}^{n-1})$ and $L^{2}(\mathbb{S}^{n-1})$. In this context we will show the following
\begin{pro} 
\label{D to N 1}
As $\e\rightarrow 0$, the operators $T_{\e}$ converge in norm to a limit operator $T_{0}$ acting between $H^{1}(\mathbb{S}^{n-1})$ and $L^{2}(\mathbb{S}^{n-1})$. Moreover the operator $T_{0}$ is determined by its values on the eigenfunctions $\phi_{j}$ of the Laplacian on $\mathbb{S}^{n-1}$, namely
\bea
T_{0}\,\phi_{j}&=&\sqrt{\tfrac{n-k}{k(n-1)}\lambda_{j}+\left(\tfrac{n-2k}{2k}\right)^{2}}\phi_{j}\hspace{0.5cm}  j\in\mathbb{N}
\eea
\end{pro}
\begin{proof} Let $\bw$ be the solution to the homogeneous problem (\ref{eq:localh}) with boundary datum $\psi = \phi_j$. Using the conformal equivariance property \eqref{confequi} on $(\g \log \e , 0 ) \times \mathbb{S}^{n-1}$, with $\g \in(0,1)$, we obtain the equation
$$
\big [ \, \mathbb{L}^0_{cyl} \big( v_\Sig (1+b)\big) \, - \, \e^{{2k}} \hbox{${n \choose k}$} \big( \tfrac{n-2k}{4k}\big)^k \tfrac{2kn}{n-2k} \big( v_{\Sig} (1+b) \big)^{\frac{2kn}{n-2k}-1}  \big] \, [ (1+b) \, \bw] \, =  \, 0.
$$
Since $b$ can be estimated in this region as $b = \bigo{\e^{2(1-\g)}}$, we have that the linear operator on the left hand side can be written as 
$$
-C_{n,k} \, v_\Sig \big[ \, \partial^{2}_{t}+\tfrac{n-k}{k(n-1)}\D_{\t}-\big(\tfrac{n-2k}{2k}\big)^{2}   
 +  \e^{2(1-\g)} \, \mathcal{P}
\big]
$$
where $\mathcal{P}$ is a linear second order partial differential operator with bounded coefficients. Using separation of variables we write $\bw$ as 
$$
\bw \,\, = \,\, \sum_{i=0}^{+\infty}\bw^i(t)\cdot\phi_{i}(\theta).
$$ 
Projecting along the $j$-th component, we obtain 
\bea
\left\{
\begin{split}
\partial^{2}_{t} \bwj -\big[\tfrac{n-k}{k(n-1)}\lambda_{j} +\big(\tfrac{n-2k}{2k}\big)^{2}\big] \, \bwj+\e^{2(1-\g)} \langle{\mathcal{P}} \, \bw,\phi_j\rangle_{L^2(\mathbb{S}^{n-1})} &= 0 \hs t\in[\g\log\e,0)\\
\bwj(0) &= 1 
\end{split}
\right.
\eea
As in previous subsection we let $\mu_i$ be the real number 
$$
\mu_{i}:=\sqrt{\tfrac{n-k}{k(n-1)}\lambda_{i} +\big(\tfrac{n-2k}{2k}\big)^{2}} \quad i \in \mathbb{N}.
$$ 
Let $\chi$ be a positive smooth non decreasing cutoff function defined on $[\log\e, 0]$, such that $\chi(t)=1$ for all $t\in[\g\log\e+1,0]$ and $\chi(t)=0$ for all $t\in[\log\e,\g\log\e]$. Multplying the equation above by 
$\chi(t)e^{\mu_{j}t}$ and integrating by parts, yields
\bea
\dt\bwj(0)-\mu_{j}&=&\int_{\g\log\e}^{0}\big[\tfrac{n-k}{k(n-1)}\lambda_{j} +\big(\tfrac {n-2k}{2k}\big)^{2}-\mu_{j}^{2}\big] \, \bwj(t)\,\chi(t)e^{\mu_{j}t}\,dt\\
&-&2\mu_{j}\int_{\g\log\e}^{\g\log\e+1}\bwj(t)\,(\dt\chi)(t)e^{\mu_{j}t}\,dt \,\, - \int_{\g\log\e}^{\g\log\e+1}\bwj(t)\,(\partial_{t}^{2}\chi)(t)e^{\mu_{j}t}\,dt\\
&+&\e^{2(1-\g)}\int_{\g\log\e}^{0}\,\langle{\mathcal{P}} \, \bw,\phi_j\rangle_{L^2(\mathbb{S}^{n-1})}  \, \chi(t)\, e^{\mu_{j}t}\,dt.
\eea
We claim that the right hand side tends to zero as $\e\rightarrow 0$. By Proposition \ref{stime:local2} we have that for every fixed $\d\in\left(-\f{n-2k}{2k},\f{n-2k}{2k}\right)$ and every $\e \in (0, \e_0]$
$$
|\bwj(t)| \, \leq  \, C\,e^{-\d t}  \quad t \in \R \,\,,
$$ 
where $C>0$ is a uniform positive constant.
Since $\mu_{j}>|\d|$, we get that there exist a positive constant $B=B(n,k,\d,\g)$ such that 
$$
|\dt\bwj(0)-\mu_{j}|\leq B\,\e^{\nu},
$$
where $\nu := \min \{ 2 (1-\g) , \g (\mu_0 - \d)   \}>0$. Now, the converge in norm of the operator $T_{\e}$ to $T_{0}$ for $\e\rightarrow 0$ follows easily. Infact, using separation of variables and writing $\psi$ as $\psi=\sum_{j=0}^{+\infty}\psi^{j}\phi_{j}(\theta)$, we get
\bea
\Vert(T_{\e}-T_{0})\,\psi\Vert^{2}_{L^{2}(\mathbb{S}^{n-1})} &=& \Vert\hbox{$\sum_{j=0}^{+\infty}$}\psi^{j}(T_{\e}-T_{0})\,(\phi_{j})\Vert^{2}_{L^{2}(\mathbb{S}^{n-1})}  \,\,\, = \,\,\, \Vert \hbox{$\sum_{j=0}^{+\infty}$}\psi^{j}\left(\dt\bwj(0)-\mu_{j}\right)\,\phi_{j}\Vert^{2}_{L^{2}(\mathbb{S}^{n-1})}\\
&=& \hbox{$\sum_{j=0}^{+\infty}$}\left|\psi^{j}(\dt\bwj(0)-\mu_{j})\right|^{2}  \,\,\, \leq \,\,\,  B^{2}\,\e^{2\nu}\Vert\psi\Vert^{2}_{H^{1}(\mathbb{S}^{n-1})}\,\,,
\eea
which ends the proof of the proposition.
\end{proof}
In the same way, we can define the Dirichlet to Neumann map for the problem
\bea
\left\{
\begin{split}
\mathbb{L}_{\bg}(u_{\e})\,\bw &= 0 \hs\Ode\\
\bw &= \psi \hs\partial\Ode
\end{split}
\right.
\eea
as
\bea
S_{\e}:\,C^{2,\b}(\mathbb{S}^{n-1})\longrightarrow C^{1,\b}(\mathbb{S}^{n-1}),\hs S_{\e}:\, \psi \longmapsto \dt{\bw_{\psi}}_{\left|_{\partial\Ode}\right.}=-\partial_{\nu}{\bw_{\psi}}_{\left|_{\partial\Ode}\right.},
\eea 
where $\partial_\nu$ is the outward normal derivative to $\Oue$, as before.
\begin{pro} 
\label{D to N 2}
As $\e\rightarrow 0$, the operators $S_{\e}$ converge in norm to a limit operator $S_{0}$ acting between $H^{1}(\mathbb{S}^{n-1})$ and $L^{2}(\mathbb{S}^{n-1})$. Moreover the operator $S_{0}$ is determined by its values on the eigenfunctions $\phi_{j}$ of the Laplacian on $\mathbb{S}^{n-1}$, namely
\bea
S_{0}\,\phi_{j}&=&-\sqrt{\tfrac{n-k}{k(n-1)}\lambda_{j}+\left(\tfrac{n-2k}{2k}\right)^{2}}\phi_{j}\hspace{0.5cm}  j\in\mathbb{N}.
\eea
\end{pro}
The proof is identical to the one of Proposition  \ref{D to N 1}.

\

\subsection{Cauchy data matching}
Let $w_{i}$ and $\bw_{i} = \bw_i(\psi)$, $i=1,2$, be the solutions to the problems
\bea
\left\{
\begin{split}
\mathbb{L}_{\bg}(u_{\e})\,w_{i} &= f \hs\Omega_{i,\e}\\
w_{i} &= 0 \hs\partial\Omega_{i,\e}
\end{split}
\right.
\quad & \quad\quad   \hbox{and} \quad & \quad \quad 
\left\{
\begin{split}
\mathbb{L}_{\bg}(u_{\e})\,\bw_{i} &= 0 \hs\Omega_{i,\e}\\
\bw_{i} &= \psi \quad \quad \, \partial\Omega_{i,\e}
\end{split}
\right.
\eea
We define the global function $w$ as
\be\label{solution}
w:=\left\{
\begin{split}
w_{1}+\bw_{1}(\psi) \hs\Oue\\
w_{2}+\bw_{2}(\psi) \hs \Ode
\end{split}
\right.
\ee
We claim that for $\e$ sufficiently small, there exists a function $\psi$, such that
$$\partial_{\nu}(w_{1}+\bw_{1}(\psi))=-\partial_{\nu}(w_{2}+\bw_{2}(\psi)).$$ This is equivalent to 
$$\partial_{\nu}\,w_{1}+\partial_{\nu}\,w_{2}=-\left(T_{\e}-S_{\e}\right)(\psi).$$
Hence, we need to invert the operator 
\bea
(T_\e - S_\e) \,\, : \,\, C^{2,\b} (\mathbb{S}^{n-1}) & \longrightarrow & C^{1,\b}(\mathbb{S}^{n-1}) 
\eea
\begin{lem}
There exists a positive real number $\e_0 = \e_0(n,k,\d)>0$ such that for every $\e \in (0, \e_0]$ and for every $\eta \in C^{1,\b}(\Sp)$ there exists a unique $\psi \in C^{2,\b}(\Sp)$ such that
$$
(T_{\e}-S_{\e})(\psi)=\eta.
$$
Moreover, there exist a positive constant $C=C(n,k,\d)>0$ such that 
$$
\Vert\psi\Vert_{C^{2,\b}(\mathbb{S}^{n-1})}\leq C\,\Vert\eta\Vert_{C^{1,\b}(\mathbb{S}^{n-1})}.
$$
\end{lem}
\begin{proof}
As a first step, for $\e$ sufficiently small, we will prove the invertibility of $(T_\e - S_\e)$ as operator acting between $H^1(\mathbb{S}^{n-1})$ and $L^2(\mathbb{S}^{n-1})$ and in this context we will provide uniform {\em a priori} estimates for solutions to 
$$
(T_{\e}-S_{\e})(\psi)=\eta.
$$
The analogous result in H\"older spaces will follow from the standard elliptic theory for first order pseudodifferential operators with bounded spectrum.

\

From Proposition \ref{D to N 1} and Proposition \ref{D to N 2} we deduce that the operators $(T_{\e}-S_{\e})$ converge in norm to the linear operator $T_{0}-S_{0}$, defined as follows
$$
(T_{0}-S_{0})(\phi_{j})=2\mu_{j}\phi_{j}  \quad \quad \mu_{j}=\sqrt{\tfrac{n-k}{k(n-1)}\lambda_{j} +\big(\tfrac{n-2k}{2k}\big)^{2}} \quad j \in \mathbb{N}
$$
where, as usual, the functions $\phi_j$'s are the eigenfunctions of $\Delta_{\mathbb{S}^{n-1}}$ and verify $- \Delta \phi_j = \lambda_j \phi_j$, $j \in \mathbb{N}$. Hence it is sufficient to show that the limit operator $(T_0 - S_0)$ is invertible and verifies the {\em a priori} estimates. Using the Fourier expansion for $\eta$, namely $\eta=\sum_{j=0}^{+\infty}\eta^{j}\phi_{j}$, we have that the ansatz for $\psi$ is given by $\sum_{j=0}^{+\infty} (\eta^{j}/2\mu_j)\phi_{j}$. Now we need to verify that this function lies in $H^1(\mathbb{S}^{n-1})$, in other words we need to test that 
$$
\hbox{$\sum_{j=0}^{+\infty}$} (1+\lambda_j) |(\eta^j/2\mu_j)|^2  <  +\infty.
$$
From the definition of the $\mu_j$'s it is straightforward to deduce that there exists a positive constant $C=C(n,k)>0$ such that $(1+\lambda_j)/4\mu_j^2 \leq C$. Thus
$$
\hbox{$\sum_{j=0}^{+\infty}$} (1+\lambda_j) |(\eta^j/2\mu_j)|^2  \leq C \, \nor{\eta}{L^2(\mathbb{S}^{n-1})}^2 .
$$
Setting $\psi := \sum_{j=0}^{+\infty} (\eta^{j}/2\mu_j)\phi_{j} \in H^1(\mathbb{S}^{n-1})$ we have that $\psi$ solves the desired equation with the estimate
$$
\nor{\psi}{H^1(\mathbb{S}^{n-1})} \leq C^{1/2} \nor{\eta}{L^2(\mathbb{S}^{n-1})}.
$$
This completes the proof.
\end{proof}

So now we can define the function $\psi$ as 
$$\psi:=(T_{\e}-S_{\e})^{-1}\left[-\partial_{\nu}w_{1}-\partial_{\nu}w_{2}\right].$$ Moreover, using the previous lemma and Proposition \ref{stime:local2}, for $\e$ sufficiently small, we have the uniform bound
\bea
\Vert\psi\Vert_{C^{2,\b}(\mathbb{S}^{n-1})}&\leq& C\left[\Vert\partial_{\nu}w_{1}\Vert_{C^{1,\b}(\mathbb{S}^{n-1})}+\Vert \, \partial_{\nu}w_{2}\Vert_{C^{1,\b}(\mathbb{S}^{n-1})}\right]\\
&\leq&  C_1 \, \e^{-\d} \big[ \, \Vert w_{1}\Vert_{C^{2,\b}_{\d}(\Oue)}+\Vert w_{2}\Vert_{C^{2,\b}_{\d}(\Ode)}\big]\\
&\leq& C_2 \, \e^{-\d} \Vert f\Vert_{C^{\,0,\b}_{\d - \frac{n-2k}{2k}(2k-1)}(M_{\e})},
\eea
where the positive constant $C_1>0$ and $C_2>0$ only depend on $n,k$ and $\d$. From this estimate, together with Proposition \ref{stime:local2} and Proposition \ref{stime:local2h} we obtain
\bea
\Vert w\Vert_{C^{2,\b}_{\d}(M_{\e})}&\leq&\Vert w_{1}\Vert_{C^{2,\b}_{\d}(\Oue)}+\Vert \bw_{1}(\psi)\Vert_{C^{2,\b}_{\d}(\Oue)}+\Vert w_{2}\Vert_{C^{2,\b}_{\d}(\Ode)}+\Vert \bw_{2}(\psi)\Vert_{C^{2,\b}_{\d}(\Ode)}\\
&\leq& C_3 \Big[ \, \Vert f\Vert_{C^{0,\b}_{\d - \frac{n-2k}{2k}(2k-1)}(M_{\e})}+ \e^\d \Vert \psi\Vert_{C^{2,\b}(\mathbb{S}^{n-1})}\Big]\\
&\leq& C_4 \Vert f\Vert_{C^{0,\b}_{\d - \frac{n-2k}{2k}(2k-1)}(M_{\e})},
\eea
where the positive constant $C_3>0$ and $C_4>0$ only depend on $n,k$ and $\d$.

\medskip

We collect all the results of this section in the following 
\begin{pro}
\label{stime:global} 
Let $\d \in \big( -\tfrac{n-2k}{2k} , \tfrac{n-2k}{2k}\big)$, then there exists a real number $\e_0= \e_0 (n,k,\d) > 0$ such that for every $\e \in (0, \e_0]$ and every $f \in C^{0,\b}(M_\e)$ there exists a unique solution $w\in C^{2,\b}(M_\e)$ to the problem
\bea
\mathbb{L}_{\bg}(u_{\e})\,[w] &=& f \hs M_\e.
\eea
Moreover, there exists a positive constant $C= C(n,k,\d)>0$ such that for every $\e\in (0,\e_0]$
\bea
\Vert w\Vert_{C^{2,\b}_{\d}(M_{\e})} \leq  C\Vert f\Vert_{C^{ \, 0,\b}_{\d - \frac{n-2k}{2k}(2k-1)}(M_{\e})}.
\eea
\end{pro}

\begin{rem}
\label{high}
We point out that using standard elliptic theory it is possible to extend the estimates above to higher order derivatives, without changing the weight parameters. The only difference is that the constant $C$ will possibly depend on the number of the derivatives involved.
\end{rem}

\

\section{Nonlinear analysis}
\label{s:nonlinear}

Now we are ready to solve the fully nonlinear equation 
\be
\label{eq:nonlinear}
\mathcal{N}_{\bg}\,(u_{\e} + w) & = & 0.
\ee
Thanks to Proposition \ref{stime:global}, which provides invertibility for the operator $\mathbb{L}{\bg}(u_\e)$,  this amounts to solve the fixed point problem
\be
\label{fixed point}
w & = & \mathbb{L}_{\bg}(u_\e)^{-1} \, \big[ \, - \mathcal{N}_{\bg}(u_\e) \, - \, \mathcal{Q}_{\bg}(u_\e) (w) \, \big] ,
\ee
where we recall that the quadratic remainder is given by
$$
\mathcal{Q}_{\bg}(u_{\e})\, (w)\,\, : = \,\, -\int_{0}^{1}\big[ \, \mathbb{L}_{\bg}(u_{\e})-\mathbb{L}_{\bg}(u_{\e}+sw) \, \big] \, [w]\,ds.
$$
We will find the fixed point $w$ as the limit of the sequence $\{w_i\}_{i\in \mathbb{N}}$ defined by means of the following Newton iteration scheme
\be
\label{newton}
\begin{cases}
\,\, w_0 & : = \quad 0\\
\,\, w_{i+1} & : = \quad  \mathbb{L}_{\bg}(u_\e)^{-1} \, \big[ \, - \mathcal{N}_{\bg}(u_\e) \, - \, \mathcal{Q}_{\bg}(u_\e) (w_i) \, \big] , \quad i \in \mathbb{N}.
\end{cases}
\ee

\subsection{Estimate of the proper error}

As a first step we estimate the {\em proper error} term 
$\NN{\bg}{u_\e}$, which is supported in the neck region $T_\e$. It is convenient to divide this region into three subdomains $T_{1,\e} := [\log \e , (2k/n) \log \e] \times \Sp$, $T_{2,\e} := [-(2k/n)\log \e, -\log \e] \times \Sp$ and $T_{\Sig,\e}:= [(2k/n)\log \e,-(2k/n)\log \e] \times \Sp$. We start by considering the proper error on $T_{1,\e}$. With the only exception of the annulus $[\log \e, \log\e +1] \times \Sp$ (where it is easy to verify that the estimate that we are going to obtain is even better), on this region we can write
$$
g_1 = u_1^{\frac{4k}{n-2k}} (1+b_1)^{\frac{4k}{n-2k}} g_{cyl}  \quad \hbox{and} \quad u_\e = u_1 \big( 1+e^{\frac{n-2k}{2k}t } \big) .
$$
Combining these two expression with the conformal equivariance property \eqref{confequi}, we obtain
\bea
\NN{\bg}{u_\e} & = & u_1^{\frac{2kn}{n-2k}} \mathcal{N}_{g_1}(1+e^{\frac{n-2k}{k}t}) \\
& = & u_1^{\frac{2kn}{n-2k}} \, \big\{\, \L{g_1}{1}{e^{\frac{n-2k}{k}t}}   \, + \, \Q{g_1}{1}{e^{\frac{n-2k}{k}t}} \, \big\} \\
& =&  (1+b_1)^{-{\frac{2kn}{n-2k}}} \big\{\, \L{cyl}{(1+b_1) u_1}{(1+b_1) u_1 e^{\frac{n-2k}{k}t}}  \, + \, \Q{cyl}{(1+b_1)u_1}{(1+b_1)u_1e^{\frac{n-2k}{k}t}} \, \big\}, 
\eea
since $\NN{g_1}{1} = 0$. Due to the fact that the coefficients of $\L{cyl}{(1+b_1) u_1}{\,\cdot \,}$ are readily estimated as $\mathcal{O}\big( \e^{\frac{n-2k}{2k}(2k-1)} e^{-\frac{n-2k}{2k}(2k-1)t}\big)$, we obtain
\bea
\NN{\bg}{u_\e} & \simeq & (\e \cosh t)^{n-2k} e^{\frac{n-2k}{k}\, t} \quad \hbox{in} \,\, T_{1,\e}.
\eea
Using the same argument it is straightforward to verify that
\bea
\NN{\bg}{u_\e} & \simeq & (\e \cosh t)^{n-2k} e^{-\frac{n-2k}{k}\, t} \quad \hbox{in} \,\, T_{2,\e}.
\eea
In the remaining region, namely $T_{\Sig,\e}$, we set $g_{\Sig,\e} := u_\e^{4k/(n-2k)} g_{cyl}$ and we write $
\bg =  (1+b)^{{4k}/({n-2k})} g_{cyl}$. From the conformal equivariance \eqref{confequi}, we obtain 
\bea
\NN{\bg}{u_\e} & = & ((1+b)u_\e^{-1})^{- \frac{2kn}{n-2k}} \mathcal{N}_{\Sig,\e}(1 +b ) \\
& = & ((1+b)u_\e^{-1})^{- \frac{2kn}{n-2k}} \, \big\{\, \NN{\Sig,\e}{1} \, + \, \L{\Sig,\e }{1}{ b }   \, + \, \Q{\Sig,\e}{1}{ b } \, \big\} \\
& =&  (1+b)^{-{\frac{2kn}{n-2k}}} \big\{ - \hbox{${n \choose k}$} \big(\tfrac{n-2k}{4k}\big)^k u_\e^{\frac{2kn}{n-2k}}   +    \L{cyl}{ u_\e}{ u_\e  b}  \, + \, \Q{cyl}{u_\e}{ u_\e  b} \, \big\}, 
\eea
since $\sigma_k(B_{\Sig,\e}) = 0$. Recalling the expression of the (homogeneous) linearized operator around a $\sigma_k$--Schwarzschild metric, we have
$$
\L{cyl}{u_\e}{u_\e  b} \, = \, -C_{n,k} \, \e^{\tfrac{n-2k}{k} (k-1)} u_\e \, \big[ \, \partial^{2}_{t}+\tfrac{n-k}{k(n-1)}\D_{\t}-\big(\tfrac{n-2k}{2k}\big)^{2}\big] \, [u_\e b] \, - \, \hbox{${n \choose k}$} \big( \tfrac{n-2k}{4k}\big)^k \tfrac{2kn}{n-2k} u_\e^{\frac{2kn}{n-2k}} b.
$$
Due to the fact that in $T_{\Sig,\e}$ one has $b=\mathcal{O}(\e^2 e^{-2t})$ and $u_\e = \mathcal{O}\big((\e\cosh t)^{(n-2k)/2k}\big)$, we infer that
\bea
\NN{\bg}{u_\e} & \simeq &  \e^{{n-2k} + 2} (\cosh t)^{\frac{n}{k}} \quad \hbox{in} \,\, T_{\Sig,\e}.
\eea
From these computations and from the definition of the weighted H\"older spaces it follows at once the following
\begin{lem}
\label{error estimate}
There exists a positive constant $A=A(n,k)>0$ such that for every $\d \in \big(-\tfrac{n-2k}{2k} , \tfrac{n-2k}{2k} \big)$ the proper error is estimated as
\bea
\nor{\, \NN{\bg}{u_\e} \,}{C^{0,\b}_{\d-\frac{n-2k}{2k}(2k-1)}(M_\e)} & \leq & A \, \e^{\frac{n-2k}{n} ( \frac{n+2k}{2k}  + \d )} .
\eea
\end{lem}


\

\subsection{Fixed point argument}

To simplify the notations of this subsection we define the two real numbers $\mu=\mu(n,k,\d)$ and $\nu= \nu(n,k,\d)$ as 
\bea
\mu \,\, := \,\, \d - \tfrac{(n-2k)(2k-1)}{2k} \\
\nu  \,\ := \,\, \tfrac{n-2k}{n} \big( \tfrac{n+2k}{2k}  + \d \big) 
\eea
since $\d$ varies in $\big(-\tfrac{n-2k}{2k} , \tfrac{n-2k}{2k}\big)$, we have that $\mu$ varies in $\big( 2k-n , - \tfrac{(n-2k)(k-1)}{k} \big)$ and $\nu$ varies in $\big( \tfrac{2(n-2k)}{n} ,\tfrac{n-2k}{k} \big)$.

\

To prove the convergence of the Newton iteration scheme \eqref{newton} we start by estimating $w_1$. Thanks to {\em a priori} estimate for the linearized equation and to the estimate of the {\em proper error} term, we immediately get
\be
\label{step1}
\nor{w_1}{C^{2,\b}_\d(M_\e)} & \leq & AC \, \e^\nu ,
\ee
where the positive constants $A=A(n,k,\d)>0$ and $C=C(n,k,\d)>0$ are the ones given in Lemma \ref{error estimate} and Proposition \ref{stime:global}, respectively. Since to achieve our goal it is important to keep track of the precise role played by these constants in the estimate, we point out that all through this section the letters $A$ and $C$ will represent the constants obtained in the estimate of Lemma \ref{error estimate} and Proposition \ref{stime:global}. 

\

We pass now to estimate the term $w_2$. From its definition it follows at once that
\be
\label{step2}
\nor{w_2}{C^{2,\b}_\d (M_\e)} & \leq & C \, \nor{\, \NN{\bg}{u_\e} \, + \, \Q{\bg}{u_\e}{w_1}   }{C^{0,\b}_\mu (M_\e)} \\
& \leq & AC \, \e^{\nu} \, + \, C \, \nor{\,  \Q{\bg}{u_\e}{w_1}   }{C^{0,\b}_\mu (M_\e)} . \nonumber
\ee
Now we need to estimate the quadratic remainder. Recalling the definition of the weighted norm we have
\bea
\nor{\,  \Q{\bg}{u_\e}{w_1}   }{C^{0,\b}_\mu (M_\e)}  & := &  \hbox{$\sum_{i=1}^2$} \nor{\,  \Q{\bg}{u_\e}{w_i}   }{C^{0,\b} (M_i \setminus B(p_i,1))} \, + \, 
\sup_{T_\e} (\e\cosh t)^\mu |\Q{\bg}{u_\e}{w_1}|  \\
& & + \, \sup_{(t,\t) \in T_\e} \bigg\{  \, (\e \cosh t)^{\mu}  \,    
\sup_{(t,\t)\neq(t',\t')} 
\frac{ |\Q{\bg}{u_\e}{w_1} (t,\t)-  \Q{\bg}{u_\e}{w_1} (t',\t')|}{|dist_{g_{\e}}((t,\t),(t',\t'))|^{\b}}
\bigg\}
\eea
The first term readily estimated as 
$$
\hbox{$\sum_{i=1}^2$} \nor{\,  \Q{\bg}{u_\e}{w_1}   }{C^{0,\b} (M_i \setminus B(p_i,1))} \,\, \leq \,\, D_0 \, \nor{w_1}{C^{2,\b}_\d(M_\e)} \,\, \leq \,\, ACD_0 \, \e^\nu ,
$$
where the positive constant $D_0>0$ only depends on $n,k$ and the $C^2$--norm of the coefficients of the metrics $g_1$ and $g_2$. We pass now to consider the term $(\e \cosh t)^\mu |\Q{\bg}{u_\e}{w_1}|$. Applying the conformal equivariance property, we get
\bea
(\e \cosh t)^\mu |\Q{\bg}{u_\e}{w_1}| 
& \leq & (\e \cosh t)^\mu \, \int_{0}^1 \big| \, \big[ \mathbb{L}_{\bg}(u_\e) \, - \, \mathbb{L}_{\bg}\big( u_\e (1+ s u_\e^{-1} w_1) \big) \, \big] \, [u_\e^{-1} w_1] \, \big| \, ds \\
& \leq & (\e \cosh t)^\mu \, \e^{\frac{n-2k}{2k}(2k-1)} \, \int_{0}^1 \big| \, \big[ \mathbb{L}_{\bg}(v_\Sig) \, - \, \mathbb{L}_{\bg} ( v_\Sig (1+ s u_\e^{-1} w_1) ) \, \big] \, [u_\e^{-1} w_1] \, \big| \, ds .
\eea
To estimate the right hand side on $T_\e$, we observe that there exists a positive constant $D_1>0$ only depending on $n$ and $k$ such that, for $j=0,1,2$, we have
$$
|\nabla^j_{\bg}(u_\e^{-1} w_1)|_{\bg} \, \leq \,D_1 \, (\e \cosh t)^{-\d - \frac{n-2k}{2k}}  \nor{w_1}{C^2_\d(M_\e)} \, \leq \, ACD_1 \,(\e \cosh t)^{-\d - \frac{n-2k}{2k}}  \, \e^{\nu} \,.
$$
Since $-\d - (n-2k)/2k < \nu$ we infer that the coefficients of the linear operator $\mathbb{L}_{\bg}(v_\Sig) \, - \, \mathbb{L}_{\bg} ( v_\Sig (1+ s u_\e^{-1} w_1) ) $ can be estimated on $T_\e$ as $\mathcal{O}\big( v_\Sig^{2k-1} \,  \, \big[\, | u_\e^{-1}w_1|  +   |\nabla_{\bg}(u_\e^{-1} w_1)|_{\bg}   +    |\nabla^2_{\bg}(u_\e^{-1} w_1)|_{\bg}   \,\big]\big) $.
We deduce that there exists a positive constant $D_2$ only depending on $n$ and $k$ such that
$$
\big| \, \big[ \mathbb{L}_{\bg}(v_\Sig) \, - \, \mathbb{L}_{\bg} ( v_\Sig (1+ s u_\e^{-1} w_1) ) \, \big] \, [u_\e^{-1} w_1] \, \big| \,\, \leq \,\, D_2 \, (\cosh t)^{\frac{n-2k}{2k}(2k-1)} \, (\e \cosh t)^{-2\d - \frac{n-2k}{2k}} \, \nor{w_1}{C^{2}_\d(M_\e)}^2.
$$
We end up with
$$
(\e \cosh t)^\mu |\Q{\bg}{u_\e}{w_1}| \,\, \leq \,\, D_3 \, (\e \cosh t)^{-\d-\frac{n-2k}{2k}} \, \nor{w_1}{C^{2}_\d(M_\e)}^2.
$$
Using the same argument one can deduce the analogous bound for the H\"older ratio, namely
$$
(\e \cosh t)^{\mu}  \,    
\sup_{(t,\t)\neq(t',\t')} 
\frac{ |\Q{\bg}{u_\e}{w_1} (t,\t)-  \Q{\bg}{u_\e}{w_1} (t',\t')|}{|dist_{g_{\e}}((t,\t),(t',\t'))|^{\b}}  \,\, \leq \,\, \, D_4 \, (\e \cosh t)^{-\d-\frac{n-2k}{2k}} \, \nor{w_1}{C^{2,\b}_\d(M_\e)}^2,
$$
for some positive constant $D_3$ and $D_4$ only depending on $n$ and $k$. Collecting these estimates one can conclude that the quadratic remainder $\Q{\bg}{u_\e}{w_1}$ verifies 
$$
\nor{\,  \Q{\bg}{u_\e}{w_1}   }{C^{0,\b}_\mu (M_\e)} \,\, \leq \,\, D  \cdot (\e \cosh t)^{-\d-\frac{n-2k}{2k}} \, \nor{w_1}{C^{2,\b}_\d(M_\e)}^2,
$$
where the positive constant $D>0$ only depends on $n,k$ and the $C^2$--norm of the coefficients of the metrics $g_1$ and $g_2$. Continuing the estimate in \eqref{step2}, we get
\bea
\nor{w_2}{C^{2,\b}_\d(M_\e)} & \leq & AC \, \e^\nu \, + \, CD \, (\e \cosh t)^{-\d-\frac{n-2k}{2k}} \, \nor{w_1}{C^{2,\b}_\d(M_\e)}^2 \\
&\leq & AC \, \e^\nu \, + \, A C^2 D \, (\e \cosh t)^{-\d-\frac{n-2k}{2k}} \, \e^\nu\, \nor{w_1}{C^{2,\b}_\d(M_\e)}\\
& = & AC \, \e^{\nu} \, + \, B \,  \nor{w_1}{C^{2,\b}_\d(M_\e)},
\eea
where in the second inequality we have used \eqref{step1} and we have set 
$$
B \, := \, A C^2 D \, (\e \cosh t)^{-\d-\frac{n-2k}{2k}} \, \e^\nu .
$$
Since $-\d - (n-2k)/2k < \nu$, there exists a real number $\e_0= \e_0 (n,k,\d,D)>0$ such that, for every $\e \in (0, \e_0]$, one can choose $B\leq\tfrac{1}{4}$. In general we obtain for every $ j\geq 1$
\bea
\nor{w_{j+1}}{C^{2,\b}_\d(M_\e)} & \leq & AC \, \e^\nu \,a_{j+1},
\eea
where the sequence $a_{j}$ is inductively defined as
$$
\begin{cases}
\,\,\,\,a_{1} \,\,:= 1\\
a_{j+1} := 1+ \tfrac{1}{4} \, a_{j}^{2}, \quad j\in\N.
\end{cases}
$$
Since $\sup_{j} a_{j} \leq 2$, one has
\bea
\nor{w_{j+1}}{C^{2,\b}_\d(M_\e)} & \leq & 2AC \, \e^\nu.
\eea

Exploiting once again the definition of the weighted norm, we obtain
\be
\label{bound}
\nor{w_{j+1}}{C^{2,\b}(M_\e)} \,\, \leq \,\,  E \, \e^{\nu-\d} \,\,, 
\ee
for some positive constant $E = E(n,k,\d,D) > 0$. From this inequality, we have that the $w_i$'s are equibounded in $C^{2,\b}(M_\e)$ and then, up to a subsequence, they converge in $C^2(M_\e)$ to a fixed point $w_{\e}$ for the problem \eqref{fixed point}. To conclude, we have that there exists a number $\e_{0}=\e_{0} (n,k,\d,D)>0$ such that for $\e \in (0,\e_0]$ the metrics 
$$
\widetilde{g}_\e  \,\, : = \,\, (1+u_\e^{-1}w_{\e})^{\frac{4k}{n-2k}} \, g_\e,
$$
where $g_\e$ are the explicit approximate solution  metrics given in Section \ref{s:as}, have positive constant $\sigma_k$--curvature equal to $2^{-k} {n \choose k}$. Finally we recall that by construction the approximate solutions metrics $g_\e$ were converging to the initial metric $g_i$ with respect to the $C^2$--topology on the compact subsets of $M_i \setminus \{p_i \}$, for $i=1,2$, as $\e\rightarrow 0$. 
On the other hand we have that 
\be
\label{small}
\nor{u_\e^{-1} \, w_\e}{C^{2}(M_\e)} & \leq & F\, \e^{\nu -\d -\frac{n-2k}{2k}} 
\ee
for some positive constant $F= F(n,k,\d,D)>0$. Since $\nu-\d-(n-2k)/2k > 0$ we have that also the exact solutions $\widetilde{g}_\e$ tend to the initial metric $g_i$ with respect to the $C^2$--topology on the compact subsets of $M_i \setminus \{p_i \}$, for $i=1,2$, as $\e\rightarrow 0$. 

\

Concerning the regularity of our solution $w_\e$, so far we have obtained that it belongs to $C^2(M_\e)$. On the other hand, as observed in Remark \ref{high}, it is possible to obtain uniform $C^{m,\b}_\d$--estimates for solutions to the linearized equation, for every $m\in\N$. Since the {\em proper error} term which appears in first step of the Newton iteration scheme is clearly smooth by construction, one can extend \eqref{bound} to
$$
\nor{w_{j+1}}{C^{m,\b}(M_\e)} \,\, \leq \,\,  E \, \e^{\nu-\d} \,\,, 
$$ 
where the positive constant $E$ may possibly depend also on $m$. The fact that $m$ is arbitrary in $\N$ implies that $w_\e$ is smooth.

\

We observe now $\widetilde{g}_{\e}$ lies in the positive cone $\Gamma_{k}^{+}$, as stated in Theorem \ref{main}. To see this fact we just need to show that $\sigma_j(\widetilde{g}_\e^{-1}\, A_{\widetilde{g}_\e}) >0$ for every $j=1, \ldots , k-1$, since $\widetilde{g}_\e$ has constant $\sigma_{k}$--curvature equal to $2^{-k} {n \choose k} $. This follows from \eqref{small} together with the fact that the approximate solutions $g_\e$'s belong to $\Gamma^+_{k-1}$, for $\e$ small enough, see Lemma \ref{SGUAN}. 

\

To conclude, we need to discuss how to remove Assumption \ref{ass}. Going through the proof it can be seen that all the analysis (uniform {\em a priori} estimate, estimate of the error term, etc.) is essentially based on blow--up techniques. For instance in the linear analysis this fact has allowed us to overcome the computational difficulty of writing down a global expression for the linearized operator about the approximate solution metrics $g_\e$ (which due to the fully nonlinear nature of our problem is rather intricate for $k>1$), letting us concentrate only on its limit behavior around the blow--up points. As a consequence one can realize that the only important features of our approximate solutions are the ones which become relevant in the limit for $\e \rightarrow 0$. It is in this limit for example that the use of the $\sigma_k$--Schwarzschild metric as a model metric on the neck reveals to be a clever choice. Having this in mind and looking at the expressions \eqref{gep} and \eqref{geps}, it is now straightforward to verify that all the limit features of our approximates solutions are not affected when the Assumption \ref{ass} is not in force, since the coefficients $a$'s which measure the discrepancy from the model metric in the general construction are of the same size of the $c$'s in \ref{gep}. This shows that the linear analysis issues still hold true in the general case. Concerning the estimate of the error term, which is the crucial step in the implementation of the Newton scheme once the uniform {\em a priori} estimates are provided, one can see reasoning as above that the only place where the {\em proper error} may possibly have a worse behavior is in the regions of the type $[\log \e, \log \e +C]\times \Sp$, for some positive constant $C>0$. In fact  the general $g_\e$'s are close enough to the model $\sigma_k$-Schwarzschild metric elsewhere and one can reproduce the desired estimate, arguing as in the proof of Lemma \ref{error estimate}. On the other hand, using the fact that $\NN{g_1}{1} \, = \, 0$, one has that
\bea
\NN{\bg}{u_\e} 
& = & u_1^{\frac{2kn}{n-2k}} \, \big\{\, \L{g_1}{1}{e^{\frac{n-2k}{k}t}}   \, + \, \Q{g_1}{1}{e^{\frac{n-2k}{k}t}} \, \big\} \,\,\, \simeq \,\,\, \e^{\frac{n-2k}{k}} \quad \quad \hbox{in} \,\, [\log \e , \log \e + C] \times \Sp .
\eea
Thus the estimates of the {\em proper error} as well are not affected by the removal of Assumption \ref{ass}, and we can definitely drop it out.
This concludes the proof of Theorem \ref{main}.

\

\section{Obstructions to the connected sum for $2k \geq n$}
\label{obstr}

We present now briefly two counterexamples to the possibility to extend Theorem \ref{main} in the case where $2k\geq n$.

\

{\bf Counterexample 1:} $n=3$, $k=2$. Let $(M_i,g_i) = (\R\mathbb{P}^3, g_{std})$, $i=1,2$, where $g_{std}$ is the standard metric on $\R\mathbb{P}^3$, i.e., the one who lifts to the round metric of $\mathbb{S}^3$. Clearly we have that $g_{std} \in \Gamma^+_2$ and has positive constant $\sigma_2$--curvature equal to $3/4$. Moreover $(\R\mathbb{P}^3, g_{std})$ is {\em non degenerate}. In fact, if $w$ is a function defined on $\R\mathbb{P}^3$ which verifies
$$ 
\L{g_{std}}{1}{ w} \,\, = \,\, 0\quad\quad\hbox{in}\,\,\R\mathbb{P}^3
$$
then it lifts to a function $\widetilde{w}$ defined on the universal cover $\mathbb{S}^3$ such that $\widetilde{w} (p) = \widetilde{w}(-p)$ for every $p \in \mathbb{S}^3$ and 
$$
\big( \, \D_{\mathbb{S}^3} + 3 \,\big) \, \widetilde{w} \,\, = \,\, 0 \quad\quad\hbox{in}\,\,\mathbb{S}^3.
$$
This clearly implies $\widetilde{w} 
\equiv 0$, since the solutions to this equation are linear combinations of the restriction to $\mathbb{S}^3$ of the coordinate functions of $\R^4$. Hence, $w \equiv 0$. At this point all the hypothesis of Theorem \ref{main} are in force, with the only exception of the inequality $2k<n$. On the other hand the connected sum $\R \mathbb{P}^3 \sharp \, \R \mathbb{P}^3$ cannot be endowed with a $2$--admissible metric. In fact such a metric would have positive Ricci curvature, as shown in \cite{gvw}, and this would contradict for instance Hamilton's theorem for $3$--manifolds \cite{hamilton}. 

\

{\bf Counterexample 2:} $n=4$, $k=2$. Let $(M_i,g_i) = (\R\mathbb{P}^4, g_{std})$, $i=1,2$, where $g_{std}$ is the standard metric on $\R\mathbb{P}^4$ as above. Clearly we have that $g_{std} \in \Gamma^+_2$ and has positive constant $\sigma_2$--curvature equal to $3/2$. The same argument as in Counterexample 1 shows that $(\R\mathbb{P}^4, g_{std})$ is {\em non degenerate}. If the Theorem \ref{main} would apply to this situation, we would end up with a locally conformally flat $2$--admissible metric $\widetilde{g}$ on the connected sum $M^4\, := \,\R \mathbb{P}^4 \sharp \, \R \mathbb{P}^4$, since the locally conformally flatness is clearly preserved by both the explicit construction of the approximate solutions and the conformal perturbation that we use to get the exact solutions. On the other hand a conformally flat $2$-admissible metric on a $4$--manifold has positive scalar curvature and fulfills the pinching conditions 
$$
2\,\, \big|\stackrel{\circ}{Ric}_{\widetilde{g}} \big| \,\, < \,\, (1/6) \,\, R_{\widetilde{g}}^2 \,\,.
$$
The Margerin's result \cite{margerin} implies now that $M^4$ is diffeomorphic to either $\mathbb{S}^4$ or $\R\mathbb{P}^4$.

\

\


\begin{thebibliography}{99}


\bibitem{aubin} T. Aubin, {\em Equations diff\'erentielles non lin\'eaires et probl\'eme de Yamabe concernant la corbure scalaire}, J. Math. Pures Appl. {\bf 55} (1976), 269--296.

\bibitem{bg} T. Branson and A. R. Gover, {\em Variational status of a class of fully nonlinear curvature prescription problems}, Calc. Var. Part. Diff. Eq. {\bf 32} (2008), 253--262.

\bibitem{cm} G. Catino and L. Mazzieri {\em Dipole type metrics with constant $\sigma_{k}$--curvature}, in preparation.

\bibitem{cgy1} S.--Y.~A.~Chang, M.~J.~Gursky and P.~C.~Yang, {\em An equation of Monge--Amp\`ere type in conformal geometry and four--manifolds of positive Ricci curvature}, Ann. of Math. {\bf 155} (2002), 709--787.

\bibitem{cgy2} S.--Y.~A.~Chang, M.~J.~Gursky and P.~C.~Yang, {\em An a priori estimate for a fully nonlinear equation on four--manifolds}, J. Anal. Math. {\bf 87} (2002), 151--186.

\bibitem{gt} D. Gilbarg and N. Trudinger, {\em Elliptic Partial Differential
Equation of Second Order}, Springer, 1983.

\bibitem{gl} M. Gromov and H. B. Lawson, {\em The classification of
simply connected manifolds of positive scalar curvature}, Ann. of
Math. {\bf 111} (1980), 423--434.

\bibitem{glw} P.~Guan, C.~S.~Lin and G.~Wang, {\em Schouten tensor and some topological properties}, Comm. Anal. Geom. {\bf 13} (2005), 887--902.

\bibitem{gvw} P.~Guan, J.~Viaclovsky and G.~Wang, {\em Some properties of the Schouten tensor and applications to conformal geometry}, Trans. Amer. Math. Soc. {\bf 355} (2003), 925--933.

\bibitem{gw} P. Guan and G. Wang, {\em A fully nonlinear conformal flow on locally conformally flat manifolds}, J. Reine Angew. Math. {\bf 557} (2003), 219--238.

\bibitem{gv2} M.~J.~Gursky and J.~Viaclovsky, {\em Fully nonlinear equations on Riemannian manifolds with negative curvature}, Indiana Univ. Math. J. {\bf 52} (2003), 399--420.

\bibitem{gv} M. Gursky and J. Viaclovsky, {\em Prescribing symmetric functions of the eigenvalues of the Ricci tensor}, Ann. Math. {\bf 166} (2007), 475--531.

\bibitem{hamilton} R.~S.~Hamilton, {\em Three--manifolds with positive Ricci curvature},
J. Diff. Geom. {\bf 17} (1982), 255--306.

\bibitem{joyce} D. Joyce, {\em Constant scalar curvature metrics on
connected sums}, Int. J. Math. Math. Sci. {\bf 7} (2003), 405--450.

\bibitem{k1} N. Kapouleas, {\em Complete constant mean curvature surfaces in Euclidean three-space}, Ann. of Math. {\bf 131} (1990), 239--330.

\bibitem{k2} N. Kapouleas, {\em Compact constant mean curvature surfaces in Euclidean three-space}, J. Diff. Geom. {\bf 33} (1991), 683--715. 

\bibitem{lp} J. M. Lee and T. H. Parker, {\em The Yamabe
Problem}, Bull. Amer. Math. Soc. {\bf 17} (1987), 37--91.

\bibitem{ll} A. Li and Y. Y. Li, {\em On some conformally invariant fully nonlinear equations}, Comm. Pure Appl. Math. {\bf 56} (2003), 1416--1464. 

\bibitem{margerin} C.~Margerin, {\em Pointwise pinched manifolds are space forms}, A. M. S. Proc. of Symp. in Pure Math. {\bf 44} (1986), 307--328.

\bibitem{mp} R. Mazzeo and F. Pacard, {\em Constant scalar curvature metrics with isolated singularities}, Duke Math. J. {\bf 99} (1999), 353--418.

\bibitem{mpu} R. Mazzeo, D. Pollack and K. Uhlenbeck, {\em Connected sums
constructions for constant scalar curvature metrics}, Topol. Method in Nonlinear Anal. {\bf 6} (1995), 207--233.

\bibitem{mazzieri1} L. Mazzieri, {\em Generalized connected sum construction
for nonzero constant scalar curvature metrics}, Comm. in Part. Diff. Eq. {\bf 33} (2008), 1--17.

\bibitem{mazzieri2} L. Mazzieri, {\em Generalized connected sum construction
for scalar flat metrics}, Manuscripta Math. {\bf 129} (2009), 137--168.

\bibitem{mn} L. Mazzieri and C. B. Ndiaye, {\em Existence of solutions for the singular $\sigma_{k}$--Yamabe problem}, preprint.

\bibitem{pacard} F. Pacard, {\em Connected sum constructions in geometry and nonlinear analysis}, http://perso-math.univ-mlv.fr/users/pacard.frank/pacard\%20\%28prepublications\%29.html.

\bibitem{schoen} R. M. Schoen, {\em Conformal deformation of a Riemannian metric to constant scalar curvature}, J. Diff. Geom. {\bf 20} (1984), 479--495.

\bibitem{schoen2} R. M. Schoen, {\em The existence of weak solutions with prescribed singular behavior for a conformally invariant scalar equation}, Comm. Pure Appl. Math., {\bf 41} (1988), 317--392.

\bibitem{sy} R. M. Schoen and S. T. Yau, {\em On the structure of manifolds with positive scalar curvature}, Manuscripta Math. {\bf 28} (1979), 159--183.

\bibitem{stw} W.--M. Sheng, N. S. Trudinger and X.--J. Wang, {\em The Yamabe problem for higher order curvatures}, J. Diff. Geom. {\bf 77} (2007), 515--553.

\bibitem{trudinger} N.S. Trudinger, {\em Remarks concerning the conformal deformation of Riemannian structures on compact manifolds}, Ann. Scuola Norm. Sup. Pisa {\bf 22} (1968), 265--274.

\bibitem{Yamabe} H. Yamabe, {\em On a deformation of Riemannian structures on compact manifolds}, Osaka Math. J. {\bf 12} (1960), 21--37.






\end{thebibliography}
\end{document}